\documentclass[11pt]{amsart}
\usepackage{amssymb,bm,mathrsfs,xcolor}
\usepackage{mathabx,mathdots}
\usepackage[all]{xy}
\usepackage{hyperref}
\newcommand{\map}[1]{\xrightarrow{#1}}
\newcommand{\mil}{\varprojlim}

\newcommand{\iso}{\cong}

\newcommand{\kk}{{\bm{k}}}
\newcommand{\inv}{\mathrm{inv}}

\newcommand{\Hom}{\mathrm{Hom}}
\newcommand{\Aut}{\mathrm{Aut}}
\newcommand{\End}{\mathrm{End}}
\newcommand{\Spec}{\mathrm{Spec}}
\newcommand{\Q}{\mathbb Q}
\newcommand{\Z}{\mathbb Z}
\newcommand{\R}{\mathbb R}
\newcommand{\C}{\mathbb C}
\newcommand{\F}{\mathbb F}
\newcommand{\A}{\mathbb A}
\newcommand{\co}{\mathcal O}
\newcommand{\alg}{\mathrm{alg}}

\newcommand{\Lie}{\mathrm{Lie}}
\newcommand{\ord}{\mathrm{ord}}
\newcommand{\GL}{\mathrm{GL}}

\newcommand{\Pic}{\mathrm{Pic}}
\newcommand{\vol}{\mathrm{vol}}
\newcommand{\Hdg}{\mathrm{Hdg}}
\newcommand{\pdec}{\circ}
\newcommand{\BKK}{\mathrm{pre}}
\newcommand{\Lat}{\mathscr{L}}

\begin{document}
\author{Benjamin Howard}
\title{Arithmetic volumes of unitary Shimura curves}
\date{}

\thanks{This research was supported in part by NSF grants  DMS-1801905 and DMS-2101636.}

\address{Department of Mathematics\\Boston College\\ 140 Commonwealth Ave. \\Chestnut Hill, MA 02467, USA}
\email{howardbe@bc.edu}

\dedicatory{To Steve Kudla, on the occasion of his 70th birthday.}

\begin{abstract}
We compute the arithmetic volumes of  integral models of  unitary Shimura curves.  
This establishes  the base case of an inductive argument to compute the arithmetic volumes of unitary Shimura varieties of higher dimension, to appear in subsequent work of Bruinier and the author. 
\end{abstract}

\maketitle
%\setcounter{section}{-1}
%\tableofcontents

\theoremstyle{plain}
\newtheorem{theorem}{Theorem}[subsection]
\newtheorem{bigtheorem}{Theorem}[section]
\newtheorem{proposition}[theorem]{Proposition}
\newtheorem{lemma}[theorem]{Lemma}
\newtheorem{corollary}[theorem]{Corollary}
\newtheorem{problem}[theorem]{Problem}

\theoremstyle{definition}
\newtheorem{definition}[theorem]{Definition}
\newtheorem{hypothesis}[theorem]{Hypothesis}

\theoremstyle{remark}
\newtheorem{remark}[theorem]{Remark}
\newtheorem{example}[theorem]{Example}
\newtheorem{question}[theorem]{Question}

\numberwithin{equation}{subsection}
\renewcommand{\thebigtheorem}{\Alph{bigtheorem}}

%%%%%%%%%%%%%%%%%%%%%%%%%%%%%%%%%%%%

\section{Introduction}

%%%%%%%%%%%%%%%%%%%%%%%%%%%%%%%%%%%%%

We compute the arithmetic volume of the integral model of a Shimura curve of type $\mathrm{GU}(1,1)$.
This calculation provides  the base case of an inductive argument to compute the arithmetic volumes of unitary Shimura varieties of higher dimension \cite{BH}.

%The strategy is to make explicit the  relation between integral models of unitary and quaternionic Shimura curves, so that we may invoke the volume calculations of Kudla-Rapoport-Yang \cite{KRY} in the quaternionic case.  

%%%%%%%%%%%%%%%%%%%%%%%%%%%%%%%%%%%%%

\subsection{Arithmetic volumes}
\label{ss:intro volume}

%%%%%%%%%%%%%%%%%%%%%%%%%%%%%%%%%%%%%

Suppose $\mathcal{X}$ is an arithmetic surface, by which we mean a regular  and separated Deligne-Mumford stack of dimension $2$, flat and of finite type over $\Z$.
The \emph{arithmetic Picard group}  $\widehat{\Pic}(\mathcal{X})$ is the group of isomorphism classes of hermitian line bundles 
\[
\widehat{\Omega} = (\Omega , \| \cdot\| )
\]
on $\mathcal{X}$. 
If   $\mathcal{X}$ is proper over $\Z$ we may identify
$
\widehat{\Pic}(\mathcal{X}) \iso \widehat{\mathrm{CH}}^1(\mathcal{X}),
$
where the right hand side is the codimension one arithmetic Chow group of Gillet-Soul\'e \cite{soule92} (extended to  Deligne-Mumford stacks  in   \cite{KRY}), and use the Arakelov intersection pairing
\[
\langle - , - \rangle : \widehat{\mathrm{CH}}^1(\mathcal{X}) \times \widehat{\mathrm{CH}}^1(\mathcal{X})
\to \R 
\]
from (2.5.13) of \cite{KRY} to define the \emph{arithmetic volume} 
$
\widehat{\mathrm{vol}}( \widehat{\Omega} ) = \langle \widehat{\Omega}  , \widehat{\Omega}  \rangle .
$

Of particular interest is the   \emph{metrized Hodge bundle}
\[
\widehat{\omega} ^\Hdg_{A / \mathcal{X}}  \in \widehat{\Pic}(\mathcal{X})
\]
associated to an abelian scheme  $\pi: A \to \mathcal{X}$  of relative dimension $d$. 
The underlying line bundle is $\omega^\Hdg_{A / \mathcal{X}} =  \pi_*   \Omega^d_{A / \mathcal{X}}$, 
where $ \Omega^d_{A / \mathcal{X}} = \bigwedge^d_{\co_A}  \Omega^1_{A / \mathcal{X}}$ is the determinant of the (locally free of rank $d$) sheaf of K\"ahler differentials on $A$, and the hermitian metric is defined by  the relation
 \begin{equation}\label{hodge metric}
\| s_x \|^2 =   \left|\frac{1}{ (2\pi i )^d}    \int_{A_x (\C)} s_x\wedge \overline{s}_x \right|
\end{equation}
for any vector $s_x \in  H^0( A_x , \Omega^d_{A_x/\C} )$  in the fiber at $x\in \mathcal{X}(\C)$.

%%%%%%%%%%%%%%%%%%%%%%%%%%%%%%%%%%%%%

\subsection{Unitary Shimura curves}

%%%%%%%%%%%%%%%%%%%%%%%%%%%%%%%%%%%%%

Let $\kk$ be a quadratic imaginary field of odd discriminant $\mathrm{disc}(\kk)$,
and let $W$ be a $\kk$-hermitian space of signature $(1,1)$ that contains a self-dual $\co_\kk$-lattice.
For any prime $p\mid \mathrm{disc}(\kk)$, abbreviate
\begin{equation}\label{alt p}
p^\pdec   =
 \begin{cases}
p &\mbox{if $W\otimes_\Q\Q_p$ is isotropic} \\
-p & \mbox{otherwise.}
\end{cases}
\end{equation}

To this data one can associate  an arithmetic surface 
\[
\mathcal{X}_W\to \Spec(\Z),
\]
 defined as a moduli space of principally polarized abelian surfaces with $\co_\kk$-actions (satisfying some additional constraints that need not concern us at the moment).
 We show  that there is a decomposition
\[
\mathcal{X}_W \iso \bigsqcup_{L \in \mathscr{L}_W}  \mathcal{C}_L
\]
into connected components  indexed by the set $\mathscr{L}_W$ of  isometry classes of self-dual $\co_\kk$-lattices $L\subset W$, and that 
\[
|  \mathscr{L}_W |  = \frac{ | \mathrm{CL}(\kk) | }{ 2^{o(\kk)-1} }
\]
where $\mathrm{CL}(\kk)$ is the ideal class group of $\kk$, and $o(\kk)$  is the number of prime divisors of $\mathrm{disc}(\kk)$; 
see  Propositions \ref{prop:steinitz orbits} and  \ref{prop:components}.
Each $\mathcal{C}_L \to \Spec(\Z)$ has geometrically connected fibers, 
and is proper  if  $W$ is anisotropic.

\begin{bigtheorem}\label{maintheorem}
Fix a connected component $\mathcal{C}_L \subset \mathcal{X}_W$, and let $A\to \mathcal{C}_L$ be the restriction of the universal abelian surface.
If $W$ is anisotropic then 
\[
\widehat{\mathrm{vol}}\big(  \widehat{\omega}^\Hdg_{A/ \mathcal{C}_L }   \big)  
  =  
 -     \deg_\C(  \omega^\Hdg_{A / \mathcal{C}_L}   )      
\left(
1+  \frac{2\zeta'(-1)}{\zeta(-1)} 
+ \frac{1}{2} \sum_{  p \mid \mathrm{disc}(\kk)     } \frac{ 1- p^\pdec  }{ 1+  p^\pdec } \cdot \log(p)
\right),
\]
where $\zeta(s)$ is the Riemann zeta function, and 
\[
\deg_\C( \omega^\Hdg_{A / \mathcal{C}_L}   )  
=  \frac{ 1 }{ 12  \cdot  | \co_\kk^\times| }    \prod_{p \mid \mathrm{disc}(\kk)   } (1 + p^\pdec ) 
\]
is the degree of  the Hodge bundle restricted to the  complex orbifold  $\mathcal{C}_L(\C)$.
\end{bigtheorem}

Theorem \ref{maintheorem} is stated in the text as Theorem \ref{thm:unitary volume}.
When $W$ is isotropic there is a similar statement, but one must use the Burgos-Kramer-K\"uhn \cite{BKK} extension of  the Gillet-Soul\'e theory on a compactification of $\mathcal{C}_L$.   See Theorem \ref{thm:compactified unitary volume}.

%%%%%%%%%%%%%%%%%%%%%%%%%%%%%%%%%%%%

\subsection{Outline of the proof}

%%%%%%%%%%%%%%%%%%%%%%%%%%%%%%%%%%%%%

We will prove Theorem \ref{maintheorem} by reducing it to the analogous result for compact quaternionic Shimura curves proved by Kudla-Rapoport-Yang \cite{KRY}, and for modular curves proved by Bost and K\"uhn \cite{kuhn}.

Suppose $\co_B \subset B$ is a maximal order in an indefinite quaternion algebra over $\Q$, and $N$ is a squarefree integer prime to $\mathrm{disc}(B)$.
In \S \ref{s:quaternionic curves} we recall the quaternionic Shimura curve
$\mathcal{X}_B(N)$ parametrizing triples $(A_0,A_1,f)$ consisting of abelian surfaces $A_0$ and $A_1$ with $\co_B$-actions, together with an $\co_B$-linear isogeny $f : A_0 \to A_1$ of degree $N^2$.   
%When $\co_B \iso M_2(\Z)$, this arithmetic surface is isomorphic to the usual (open) modular curve of level $\Gamma_0(N)$.
To each level $N$ Eichler order $R \subset \co_B$ we associate an abelian surface with $R$-action
\[
A_R \to \mathcal{X}_B(N)
\]
in such a way that the universal $f:A_0 \to A_1$ factors as a composition
\[
A_0 \to A_R \to A_1
\]
of $R$-linear isogenies of degree $N$.  We call  $A_R$ the \emph{intermediate abelian surface} determined by the Eichler order $R$.
The main result of  \S \ref{s:quaternionic curves}   is  Theorem \ref{thm:intermediate volume}, which computes 
the arithmetic volume of   the metrized Hodge bundle  of $A_R$ by comparing it to that of $A_0$, which is known by the works cited above.  

In \S \ref{s:hermitian} we explain how the hermitian space $W$ of signature $(1,1)$ determines an indefinite quaternion algebra $B$ with an embedding $\kk \to B$, and  how the self-dual $\co_\kk$-lattice $L\subset W$ determines an Eichler order $R\subset B$ of level   
\[
N=-\mathrm{disc}(\kk) / \mathrm{disc}(B).
\]

In    \S \ref{s:unitary curves} we construct   a finite \'etale surjection 
$
q_L : \mathcal{X}_B(N) \to \mathcal{C}_L
$
of degree $| \co_\kk^\times|$ in such a way that the pullback of the universal  $A \to \mathcal{C}_L$ is 
precisely the intermediate abelian surface $A_R$ of the Eichler order determined by $L$.

As explained at the beginning of \S \ref{s:final volumes},  Theorem \ref{maintheorem} follows from the above constructions.
 The remainder of \S \ref{s:final volumes} is devoted to extending that theorem to the case of isotropic $W$.

\subsection{Acknowledgements}

The author thanks the anonymous referee for helpful comments and suggestions.

%%%%%%%%%%%%%%%%%%%%%%%%%%%%%%%%%%%%

\section{Quaternionic Shimura curves}
\label{s:quaternionic curves}

%%%%%%%%%%%%%%%%%%%%%%%%%%%%%%%%%%%%%

Over a quaternionic Shimura curve with level structure there is a  universal isogeny $f:A_0\to A_1$ of abelian surfaces, each equipped  with an action of a fixed maximal order $\co_B$ in an indefinite quaternion algebra.
The main result of this section is the computation of the arithmetic volume of the metrized Hodge bundle of a third abelian surface, lying between $A_0$ and $A_1$,  and depending on a choice of Eichler order in $\co_B$.

%%%%%%%%%%%%%%%%%%%%%%%%%%%%%%%%%%%%

\subsection{The moduli problem}
\label{ss:drinfeld moduli}

%%%%%%%%%%%%%%%%%%%%%%%%%%%%%%%%%%%%%

By a  \emph{quaternion algebra} we mean a central simple $\Q$-algebra  of dimension $4$.  
Any quaternion algebra  $B$ admits a $\Q$-basis $\{1,i,j, ij\}$ with $i^2=a$, $j^2=b$, and $ij = -ji$, for some $a,b\in \Q^\times$.
We write
\[
B \iso  \left(\frac{ a,b} {\Q } \right) 
\]
to indicate the existence of such a basis.
For any place $p\le \infty $ of $\Q$ the \emph{local invariant} is defined by
\[
\inv_p(B) = ( a,b )_p,
\]
where the right hand side is the Hilbert symbol.  
Two quaternion algebras are isomorphic if and only of their local invariants agree at all places.  
As usual,  $B$ is \emph{indefinite} if  $\inv_\infty(B) = 1$.

%
% The \emph{main involution} $b\mapsto \overline{b}$ on $B$   is  the  unique $\Q$-algebra anti-automorphism satisfying  $ \overline{i} = -i$ and $\overline{j} = -j$ for all such basis. 
%  The \emph{reduced trace} and \emph{reduced norm} on $B$ are 
% \[
% \mathrm{Trd}(b) = b+\overline{b} \quad\mbox{and}\quad  \mathrm{Nrd}(b) = b\overline{b}.
% \]
%
%

Fix  an indefinite quaternion algebra $B$ and a maximal order $\co_B \subset B$.       
The maximal order is unique up to $B^\times$-conjugacy, and so the particular choice is unimportant.
Let $b\mapsto \overline{b}$ denote the main involution on $B$, and denote by 
$\mathrm{Trd}(b) = b+\overline{b}$ and  $\mathrm{Nrd}(b) = b\overline{b}$ the reduced trace and reduced norm.

\begin{definition}\label{def:OBsurface}
 An  \emph{$\co_B$-abelian surface} over a $\Z$-scheme $S$ is an abelian scheme $A \to S$ of relative dimension two, together with a ring homomorphism $\co_B \to \End(A)$ satisfying Drinfeld's determinant condition:
every  $b\in \co_B$ acts on the locally free $\co_S$-module $\Lie(A)$ with characteristic polynomial 
\[
x^2- (b+ \overline{b}) x  + b \overline{b} \in \Z [x],
\]
viewed as a polynomial in $\co_S[x]$.  
 \end{definition}
 
\begin{remark}
The determinant condition stated here is different in appearance from  Drinfeld's \emph{special} condition, as defined  in  Remark III.3.3 of \cite{boutot-carayol}. The equivalence of the two conditions can be easily proved using Proposition 2.1.3 of \cite{hartwig}.
\end{remark}

Fix  a squarefree integer $N>0$ prime to $\mathrm{disc}(B)$.
 Let $\mathcal{X}_B(N)$  be the Deligne-Mumford stack whose functor of points assigns to a $\Z$-scheme $S$ the groupoid of 
 tuples $(A_0,A_1, f)$ in which  $A_0$ and $A_1$ are $\co_B$-abelian surfaces over $S$, and 
$
f : A_0 \to A_1
$
 is an $\co_B$-linear isogeny of degree $N^2$ whose kernel is contained in the $N$-torsion subgroup scheme $A_0[N] \subset A_0$.

It is known that $\mathcal{X}_B(N)$ is an arithmetic surface, in the sense of \S \ref{ss:intro volume}, with geometrically connected fibers.    It is smooth outside  characteristics dividing $N \mathrm{disc}(B)$. 
 If $B$ is a division algebra then $\mathcal{X}_B(N)$ is proper over $\Z$. 
 For all of this, see   \cite{boutot-carayol} and  \cite{buzzard}.

\begin{definition}\label{def:AL}
As in the classical  theory of modular curves,  the \emph{Atkin-Lehner involution} 
\[
 \tau : \mathcal{X}_B(N) \to \mathcal{X}_B(N)
\]
is the  automorphism  sending $f:A_0 \to A_1$ to the isogeny $f^\vee : A_1 \to A_0$  characterized by 
$f^\vee \circ f =[N]$ and $f\circ f^\vee=[N]$.
\end{definition}

From now on we let
$
f: A_0 \to  A_1 
$
 denote the universal object over $\mathcal{X}_B(N)$, and fix a level $N$ Eichler order
 $R\subset \co_B$.  The finite flat group scheme
$
\mathrm{ker}(f) \to \mathcal{X}_B(N)
$
has order $N^2$, and  carries an action of 
 \begin{equation}\label{matrix coordinates}
 \co_B/ N\co_B \iso M_2(\Z/N\Z).
 \end{equation}
We choose  \eqref{matrix coordinates} in such a way that $R/N \co_B$ is identified with the upper triangular matrices, and 
 define orthogonal idempotents in \eqref{matrix coordinates} by
\[
e=\left( \begin{matrix} 1&0 \\ 0 & 0 \end{matrix} \right) 
\quad \mbox{ and }\quad 
e'=\left( \begin{matrix} 0&0 \\ 0 & 1 \end{matrix} \right).
\]
One can easily check that $C_R = e \cdot \mathrm{ker}(f)$  is characterized as the unique $R$-stable finite flat subgroup scheme  of $\mathrm{ker}(f)$  of order $N$.  In particular, while it depends on the Eichler order $R\subset \co_B$, it does not depend on the isomorphism \eqref{matrix coordinates}.

\begin{definition}\label{def:intermediate}
The quotient 
 $
 A_R = A_0 /C_R
 $
   is  the \emph{intermediate abelian surface} determined by the Eichler order $R$.  
It is an abelian surface over $\mathcal{X}_B(N)$ endowed  with an action  $R\to \End(A_R)$, and with $R$-linear  isogenies
\[
A_0 \map{\phi} A_R \map{\psi} A_1 
\]
each of degree $N$.
\end{definition}

The following lemma will be needed in the proof of Theorem \ref{thm:intermediate volume}.

\begin{lemma}\label{lem:intermediate fixed point}
There is an isogeny $\tau^* A_R \to A_R$ of degree prime to $N$.
\end{lemma}

\begin{proof}
If we set $C_R' = e'\cdot \mathrm{ker}(f)$, there is a chain of finite flat group schemes
\[
0 \subset  C_R  \subset C_R \oplus C_R' \subset e A_0[N] \oplus C_R' \subset A_0[N]
\]
of orders $1$, $N$, $N^2$, $N^3$, and $N^4$,  respectively.
Taking the quotients of $A_0$ by these subgroups yields a 
chain of degree $N$ isogenies
\[
A_0 \to A_R \to A_1 \to  \frac{A_0}{e A_0[N] \oplus C_R'} \to A_0
\]
whose composition is $[N]$.  
The composition of the first two arrows is $f$, hence the   composition of the second two arrows is $f^\vee$.
Thus 
\[
\tau^* A_R \iso  \frac{A_0}{e A_0[N] \oplus C_R'},
\]
as both sides are characterized as the unique $R$-stable  finite flat subgroup scheme of  $\ker(f^\vee)$ of order $N$.

Recalling the fixed isomorphism \eqref{matrix coordinates}, choose any lift $\gamma \in \co_B$ of 
\[
 \begin{pmatrix} 0 & 1 \\ 0 &0   \end{pmatrix}  \in  M_2(\Z/N\Z) \iso \co_B/N \co_B
\]
of reduced norm  $\gamma \overline{\gamma} = N  M$ with $M$ an  integer coprime to $N$.
  The degree $(NM)^2$ endomorphism $\gamma \in \End(A_0)$ descends to a degree $M^2$ isogeny
\[
\tau^* A_R \iso \frac{ A_0 }{  eA_0[N] \oplus C_R' } \map{\gamma} A_0 / C_R \iso A_R. \qedhere
\]
\end{proof}

\begin{remark}\label{rem:split intermediate}
  If $B\iso M_2(\Q)$,  we may choose an isomorphism  $\co_B\iso M_2(\Z)$ in such a way that 
  \[
R \iso \left\{  \left( \begin{matrix}  a & b \\ Nc & d \end{matrix} \right)  : a,b,c,d \in \Z  \right\}.
\]
 Such a choice determines an isomorphism  from the   moduli space of elliptic curves with $\Gamma_0(N)$-structure to  $\mathcal{X}_B(N)$, under which  $f:A_0 \to A_1$ is identified with the square
\begin{equation}\label{elliptic isogeny}
\varphi \times \varphi   : E_0 \times E_0 \to E_1\times E_1, 
\end{equation}
of the universal cyclic $N$-isogeny elliptic curves.  Under this identification, 
\[
C_R = \ker(\varphi) \times \{0\},
\]
the intermediate abelian surface is  $A_R=E_1 \times E_0$, and $\tau^*A_R = E_0 \times E_1$. 
The isogeny 
\[
E_0 \times E_1  = \tau^*A_R \to A_R = E_1 \times E_0
\]
determined by the choice of 
\[
\gamma =  \begin{pmatrix} 0 & 1 \\ N &0   \end{pmatrix}  \in  M_2(\Z)
\]
in the proof of Lemma \ref{lem:intermediate fixed point}  is the isomorphism $(x_0,x_1)\mapsto (x_1,x_0)$.
\end{remark}

%%%%%%%%%%%%%%%%%%%%%%%%%%%

\subsection{Comparison of hermitian  line bundles}

%%%%%%%%%%%%%%%%%%%%%%%%%%%

We continue with the notation of the previous subsection.  
Our goal is to compare the metrized Hodge bundle of the universal $A_0 \to \mathcal{X}_B(N)$ with that of the intermediate abelian surface $A_R \to \mathcal{X}_B(N)$ determined by an Eichler order $R\subset \co_B$ of level $N$.

By results of Katz-Mazur \cite{katz-mazur}, extended to quaternionic Shimura curves in  \cite{buzzard}, 
 the reduction of $\mathcal{X}_B(N)$ at a prime $p\mid N$ is a union 
\[
\mathcal{X}_B(N)_{ \F_p} = \mathcal{F}_p \cup \mathcal{V}_p
\]
of two irreducible components, interchanged by the Atkin-Lehner involution of Definition \eqref{def:AL}.
Each has a natural structure of reduced closed substack, and the fiber product 
\begin{equation}\label{sslocus}
\mathcal{S}_p = \mathcal{F}_p \times_{ \mathcal{X}_B(N) } \mathcal{V}_p 
\end{equation}
is the reduced locus of supersingular points.

One distinguishes between $\mathcal{F}_p$ and $\mathcal{V}_p$ as follows: after restricting to the two components, the universal isogeny $f: A_0 \to A_1$ admits factorizations 
\begin{equation}\label{Fcomponent}
A_{ 0 / \mathcal{F}_p} \map{ \mathrm{Fr} } A^{(p)}_{0/ \mathcal{F}_p} \to A_{ 1 / \mathcal{F}_p } 
\end{equation}
and
\begin{equation}\label{Vcomponent}
A_{ 0 / \mathcal{V}_p } \to A^{(p)}_{ 1 / \mathcal{V}_p}  \map{ \mathrm{Ver} }   A_{ 1 / \mathcal{V}_p} ,
\end{equation}
where $\mathrm{Fr}$ and $\mathrm{Ver}$ are the usual Frobenius and Verschiebung morphisms, and the unlabelled arrows are isogenies of degree $(N/p)^2$.

By the regularity of $\mathcal{X}_B(N)$, the vertical Weil divisors $\mathcal{F}_p$ and $\mathcal{V}_p$  are each defined locally by a single equation, and so determine line bundles  $\co(\mathcal{F}_p)$ and $\co(\mathcal{V}_p)$ on $\mathcal{X}_B(N)$.
Noting that these line bundles are canonically trivialized in the complex fiber by the constant function $1$,  we endow each with the constant metric  $  \| 1\|  = \sqrt{p}$.  The resulting 
hermitian line bundles are denoted
\begin{equation}\label{vertical bundles}
\widehat{\mathcal{F}}_p  , \widehat{\mathcal{V}}_p  \in \widehat{\Pic}( \mathcal{X}_B(N) ).
\end{equation}
As the tensor product $\co(\mathcal{F}_p) \otimes \co(\mathcal{V}_p) \iso \co( \mathcal{X}_B(N)_{\F_p} )$ is  trivialized by the global section $p^{-1}$ of constant norm $1$, we have the relation
\begin{equation}\label{level component cancel}
\widehat{\mathcal{F}}_p + \widehat{\mathcal{V}}_p   =0
\end{equation}
in the arithmetic Picard group.

\begin{proposition}\label{prop:intermediate hodge compare}
The metrized Hodge bundles of $A_0$ and $A_R$ satisfy 
\[
 \widehat{\omega}^\Hdg_{A_0/ \mathcal{X}_B(N)}
 = \widehat{\omega}^\Hdg_{A_R/ \mathcal{X}_B(N)}
 + \sum_{p\mid N} \widehat{\mathcal{F}} _p  \in \widehat{\Pic}( \mathcal{X}_B(N) )  .
\]
\end{proposition}

\begin{proof}
The canonical  isogeny $\phi : A_0 \to A_R$ induces a morphism 
\begin{equation}\label{hodge shift}
\phi^*  :  \omega^\Hdg_{A_R/ \mathcal{X}_B(N)} \to  \omega^\Hdg_{A_0/ \mathcal{X}_B(N)}
\end{equation}
of Hodge bundles.  If we view  this as a section 
\begin{equation}\label{hodge shift section}
\phi^* \in H^0\big( \mathcal{X}_B(N) , \omega^\Hdg_{A_0/ \mathcal{X}_B(N)} \otimes ( \omega^\Hdg_{A_R/ \mathcal{X}_B(N)})^{-1} \big), 
\end{equation}
then unpacking the definition of the metric \eqref{hodge metric} shows that
\begin{equation}\label{hodge shift infinite}
\| \phi^* \| =  \sqrt{ \deg(\phi)} = \sqrt{N} .
\end{equation}

The morphism of vector bundles $\phi : \Lie(A_0) \to \Lie(A_R)$ restricts to an isomorphism over $\mathcal{X}_B(N)_{/\Z[1/N]}$, and hence so does  \eqref{hodge shift}. From this it follows that 
\[
\mathrm{div}( \phi^*  ) = \sum_{p\mid N} a_p \mathcal{F}_p +\sum_{p\mid N} b_p \mathcal{V}_p
\]
for  nonnegative  $a_p, b_p \in \Z$.
We will show that  all $a_p=1$ and  all $b_p=0$.

Suppose 
\[
\Spec(k) \map{x} \mathcal{X}_B(N)_{ \F_p}^\ord = \mathcal{X}_B(N)_{ \F_p} \smallsetminus \mathcal{S}_p
\]
 is a geometric point of the ordinary locus.
For  $i \in \{ 0,1 \}$ the $p$-divisible group  of  $A_{i,x}$  sits in a connected-\'etale sequence
\[
0 \to H^0_{i,x} \to A_{i,x} [p^\infty] \to H^{\mathrm{et}}_{i,x} \to 0,
\]
and the connected part satisfies
\begin{equation}\label{ordinary factors}
H^0_{i,x} \iso  \mu_{p^\infty/k} \times  \mu_{p^\infty/k}.
\end{equation}
The isogeny $f : A_0 \to A_1$ induces an isogeny of $p$-divisible groups
\begin{equation}\label{connected isogeny}
f_x^0 : H^0_{0,x} \to H^0_{1,x} .
\end{equation}

If the geometric point $x$ lies in $\mathcal{V}_p^\mathrm{ord} = \mathcal{V}_p \smallsetminus \mathcal{S}_p$,
 then  \eqref{Vcomponent} implies that the  isogeny \eqref{connected isogeny} factors as
\[
H^0_{0,x} \to H^{0,(p)}_{1,x} \map{ \mathrm{Ver} }   H^0_{1,x}.
\]
The first arrow is an isomorphism because the first arrow in  \eqref{Vcomponent} has degree prime to $p$,
and the second arrow is an isomorphism by \eqref{ordinary factors}.
Hence \eqref{connected isogeny} is an isomorphism. 
As $\Lie(H^0_{i,x}) \iso \Lie(A_{i,x})$, we deduce that $f_x : A_{0,x} \to A_{1,x}$ induces an isomorphism
 $ \Lie(A_{0,x}) \to \Lie(A_{1,x})$.
This last map factors as 
  \[
  \Lie(A_{0,x})  \map{\phi_x}\Lie(  A_{R,x} )  \to  \Lie( A_{1,x}),
  \]
 and so the first arrow in the composition is also an isomorphism. 
 It follows  that the section \eqref{hodge shift section} is nowhere vanishing on $\mathcal{V}_p^\ord$, and so  $b_p=0$.

Suppose instead that the geometric point $x$ lies in $\mathcal{F}_p^\ord = \mathcal{F}_p \smallsetminus \mathcal{S}_p$.
In this case 
 \eqref{Fcomponent} implies that the  isogeny \eqref{connected isogeny} factors as
\[
H^0_{0,x}  \map{ \mathrm{Fr} }  H^{0,(p)}_{0,x}  \to   H^0_{1,x}.
\]
The second arrow is an isomorphism because the second arrow in  \eqref{Fcomponent} has degree prime to $p$,
while \eqref{ordinary factors} implies that the kernel of the first arrow is  the $p$-torsion subgroup of $H^0_{0,x}$.
Thus the isomorphisms \eqref{ordinary factors} may be chosen so that 
\begin{equation}\label{torus maps}
\mu_{p^\infty/k} \times  \mu_{p^\infty/k} \iso H^0_{0,x} \map{f_x^0} H^0_{1,x} \iso \mu_{p^\infty/k} \times  \mu_{p^\infty/k}
\end{equation}
is multiplication by $p$.  Again using $\Lie(H^0_{i,x}) \iso \Lie(A_{i,x})$, we now find that the isogeny $f:A_0 \to A_1$ induces the trivial map on Lie algebras.  Thus \eqref{hodge shift section} vanishes identically along $\mathcal{F}_p^\ord$, and so $a_p>0$.

To go further, let  $W = W(k)$ be the Witt ring.   Taking the Serre-Tate canonical lifts of $A_{0,x}$ and $A_{1,x}$  to $W$ provides us with a diagram
\[
\xymatrix{
{  \Spec(k)  }  \ar[r]^{x}   \ar[d]  & { \mathcal{X}_B(N) }    \\ 
 {   \Spec( W ) . }  \ar[ur]_{\widetilde{x}}  &  
}
\]
If we denote by $H^0_{i,\tilde{x}}$ the connected part of the $p$-divisible group $A_{i,\tilde{x}}[p^\infty]$ over $W$, the
Grothendieck-Messing deformation theory  implies that the maps in \eqref{torus maps} lift uniquely to 
\begin{equation}\label{formal coordinates}
\mu_{p^\infty/W} \times  \mu_{p^\infty/W} \iso H^0_{0,\tilde{x}} \map{f_{\tilde{x}}^0} H^0_{1, \tilde{x}} 
\iso \mu_{p^\infty/W} \times  \mu_{p^\infty/W},
\end{equation}
and this composition is still multiplication by $p$.  Taking Lie algebras, we find that 
\[
W^2 \iso  \Lie(H_{0,\tilde{x}}) \iso  \Lie(A_{0,\tilde{x}}) \map{f_{\tilde{x}}}  \Lie(A_{1,\tilde{x}}) \iso   \Lie(H_{0,\tilde{x}}) \iso W^2
\]
is multiplication by $p$, and so each map in the factorization 
\[
 \Lie(A_{0,\tilde{x}}) \map{\phi}  \Lie(A_{R,\tilde{x}})  \map{\psi} \Lie(A_{0,\tilde{x}})
\]
of $f_{\tilde{x}}$ has cokernel isomorphic to $W/pW$.
It follows that the restriction of \eqref{hodge shift} to a morphism of rank one $W$-modules
\[
\phi^* :  \omega^\Hdg_{A_{0,\tilde{x}} /W } \to  \omega^\Hdg_{A_{R,\tilde{x}} /W } 
\]
 has cokernel isomorphic to $W/pW$, and from this one deduces  $a_p=1$.

We have now shown that  \eqref{hodge shift section} has divisor $\sum_{p\mid N} \mathcal{F}_p$
and norm  \eqref{hodge shift infinite}, and the claim follows.
\end{proof}

%%%%%%%%%%%%%%%%%%%%%%%%%%%

\subsection{Arithmetic volumes}

%%%%%%%%%%%%%%%%%%%%%%%%%%%

We continue to let $R\subset \co_B$ be an Eichler order of squarefree level $N$, prime to $\mathrm{disc}(B)$, and assume that $B$ is a division algebra.  Hence $\mathcal{X}_B(N)$ is proper over $\Z$.

We now combine results of Kudla-Rapoport-Yang on the metrized Hodge bundle of $A_0 \to \mathcal{X}_B(N)$ with  Proposition  \ref{prop:intermediate hodge compare} to obtain  results on the Hodge bundle of the intermediate abelian surface of Definition \ref{def:intermediate}.

\begin{theorem}\label{thm:intermediate volume}
The Hodge bundle of  $A_R \to \mathcal{X}_B(N)$ has geometric degree 
\[
 \deg_\C(  \omega^\Hdg_{A_R/ \mathcal{X}_B(N)}   ) 
=  \frac{1}{12} \prod_{p  \mid \mathrm{disc}(B)} (1-p)   \prod_{p \mid N} (1+p)
\]
and  arithmetic volume
\begin{align*}\lefteqn{
\widehat{\mathrm{vol}}  (  \widehat{ \omega}^\Hdg_{A_R/ \mathcal{X}_B(N) }   ) 
   =  - \deg_\C(  \omega^\Hdg_{A_R/ \mathcal{X}_B(N)}   )   } \\
& \qquad \times
\left(
1+ \frac{2\zeta'(-1)}{\zeta(-1)} 
+  \sum_{ p\mid \mathrm{disc}(B)} \frac{1+p}{1-p} \cdot  \frac{ \log(p)}{2}
+  \sum_{ p\mid N} \frac{1-p}{1+p} \cdot  \frac{ \log(p)}{2}
\right).
\end{align*}
\end{theorem}

\begin{proof}
The metrized Hodge bundle  of the  universal  $A_0 \to \mathcal{X}_B(N)$ has  geometric degree
\begin{equation}\label{A_0 deg}
 \deg _\C  (  \omega^{\Hdg}_{A_0/ \mathcal{X}_B(N) }   )      
=  \frac{1}{12} \prod_{p  \mid \mathrm{disc}(B)} (1-p)  \prod_{p \mid N} (1+p)
\end{equation}
and arithmetic volume
\begin{align}\lefteqn{
 \widehat{\mathrm{vol}}   (  \widehat{\omega}^{\Hdg}_{A_0/ \mathcal{X}_B(N) }   )   = }   \label{A_0 vol}     \\
&     - \deg _\C  (  \omega^{\Hdg}_{A_0/ \mathcal{X}_B(N) }   )    \left(
1 +  \frac{2 \zeta'(-1)}{\zeta(-1)}  +  \sum_{ p\mid \mathrm{disc}(B)} \frac{1+p}{1-p} \cdot \frac{ \log(p) }{2}
\right).     \nonumber
\end{align}
When $N=1$  this is Corollary 7.11.2 of   \cite{KRY}, adjusted to account for the fact that our
metrized Hodge bundle differs from the $\widehat{\omega}$ appearing there by a rescaling of the metric; see especially (3.3.4) of \emph{loc.~cit.}.
The forgetful morphism 
\[
\mathcal{X}_B(N) \to \mathcal{X}_B(1)
\]
 is finite  flat of degree $\prod_{p\mid N}(1+p)$  by Theorem 4.7 of \cite{buzzard}, and    the  formulas at level $N$   follow from those at level $1$  (for the arithmetic volume use  the projection formula of Section III.3.1 of \cite{soule92}).

It is immediate from Proposition \ref{prop:intermediate hodge compare} that the metrized Hodge bundles of $A_R$ and $A_0$  are isomorphic in the complex fiber,  and so the first claim of the theorem follows from \eqref{A_0 deg}.

Fix a prime $p\mid N$, and  recall that the  Atkin-Lehner involution $\tau$ of Definition  \eqref{def:AL}  interchanges  $\mathcal{F}_p$ and $\mathcal{V}_p$.  Using this and  \eqref{level component cancel} we find that 
\[
\tau^* \widehat{\mathcal{F}}_p = \widehat{\mathcal{V}}_p = - \widehat{\mathcal{F}}_p
\]
in the arithmetic Chow group, and so  the invariance of $\widehat{\deg}$ under $\tau^*$ implies
\begin{equation}\label{cross term a}
 \widehat{\deg} \big( \widehat{\mathcal{F}}_p  \cdot \widehat{\omega}^\Hdg_{A_R/ \mathcal{X}_B(N)}  \big)
=- \widehat{\deg} \big( \widehat{\mathcal{F}}_p  \cdot \tau^* \widehat{\omega}^\Hdg_{A_R/ \mathcal{X}_B(N)}  \big).
\end{equation}
On the other hand, any isogeny as in Lemma \ref{lem:intermediate fixed point} induces  a morphism
\[
\omega^\Hdg_{A_R / \mathcal{X}_B(N) }  \to \tau^* \omega^\Hdg_{A_R / \mathcal{X}_B(N) }   
\]
that restricts to an \emph{isomorphism} over the  mod $p$ fiber of $\mathcal{X}_B(N)$, and in particular  over $\mathcal{F}_p$. 
In general, we have
\[
 \widehat{\deg} \big( \widehat{\mathcal{F}}_p  \cdot \widehat{\Omega}  \big)
 = \deg  ( \Omega \mid_{  \mathcal{F}_p  }  )  \cdot  \log(p) 
\]
for any hermitian line bundle $\widehat{\Omega}$ on $\mathcal{X}_B(N)$, and so the existence of such an isomorphism implies 
\begin{equation}\label{cross term b}
 \widehat{\deg} \big( \widehat{\mathcal{F}}_p  \cdot \widehat{\omega}^\Hdg_{A_R/ \mathcal{X}_B(N)}  \big)
= 
\widehat{\deg} \big( \widehat{\mathcal{F}}_p  \cdot \tau^* \widehat{\omega}^\Hdg_{A_R/ \mathcal{X}_B(N)}  \big).
\end{equation}
Combining \eqref{cross term a} and \eqref{cross term b} shows 
\begin{equation}\label{no cross term}
\widehat{\deg} \big( \widehat{\mathcal{F}}_p  \cdot \widehat{\omega}^\Hdg_{A_R/ \mathcal{X}_B(N)}  \big) = 0 .
\end{equation}

It follows from (\ref{level component cancel}) that 
\[
- \widehat{\mathcal{F}}_p  \cdot \widehat{\mathcal{F}}_p  =   \widehat{\mathcal{V}}_p \cdot \widehat{\mathcal{F}}_p 
=   ( \mathcal{S}_p ,0 ) \in \widehat{\mathrm{CH}}^2( \mathcal{X}_B(N) ),
\]
where $(\mathcal{S}_p , 0)$ is the reduced supersingular locus 
 \eqref{sslocus} endowed with the trivial Green current.  Hence
\begin{align}
 - \widehat{\vol} \big( \widehat{\mathcal{F}}_p   \big) 
& = \sum_{     x\in\mathcal{S}_p(\F_p^\alg)   }   \frac{\log(p)}{ | \Aut(x) |}    \nonumber \\
& = 
\frac{ \log(p)}{24}     \left( \frac{ p-1}{p+1} \right)   
\prod_{ \ell \mid \mathrm{disc}(B)  }(   \ell -1  )   \prod_{ \ell \mid N   }(  \ell  +1) ,  \label{mass formula}
\end{align}
where the second equality is the  Eicher-Deuring mass formula, extended to quaternionic Shimura curves by Yu \cite{yu}.

Combining  Proposition \ref{prop:intermediate hodge compare} with  \eqref{no cross term} shows that
\[
\widehat{\mathrm{vol}}  \big( \widehat{\omega}^\Hdg_{A_0 / \mathcal{X}_B(N)} \big)   = 
  \widehat{\mathrm{vol}}   \big(  \widehat{\omega}^\Hdg_{A_R / \mathcal{X}_B(N)} \big)      
+  \sum_{p\mid N}  \widehat{\vol} (  \widehat{\mathcal{F}} _p   ),
\]
and combining this with \eqref{A_0 vol} and  \eqref{mass formula}  proves the second claim of the theorem.
  \end{proof}

%%%%%%%%%%%%%%%%%%%%%%%%%%%%%%%%%%%%

\section{Hermitian spaces and lattices}
\label{s:hermitian}

%%%%%%%%%%%%%%%%%%%%%%%%%%%%%%%%%%%%%

Let $\kk$ be a quadratic imaginary field of discriminant $\mathrm{disc}(\kk)$.
Eventually we will specialize to the case where the discriminant is odd.

%%%%%%%%%%%%%%%%%%%%%%%%%%%%%%%%%%%%

\subsection{Quaternion embeddings}

%%%%%%%%%%%%%%%%%%%%%%%%%%%%%%%%%%%%%

Suppose  $W$ is a hermitian space over $\kk$, which we always assume to be finite dimensional and non-degenerate, and denote by $\langle-,-\rangle$ the  hermitian form.
At each place $p\le \infty$ we define the \emph{local invariant}
\[
\inv_p(W) = (  \det(W) ,  \mathrm{disc}(\kk) )_p \in \{ \pm 1\},
\]
where the right hand side is the Hilbert symbol, and $\det(W) \in \Q^\times / \mathrm{Nm}(\kk^\times)$ is the determinant of $( \langle x_i,x_j \rangle)$ for any   $\kk$-basis of $x_1,\ldots, x_d\in W$.
Two $\kk$-hermitian spaces of the same dimension are isomorphic if and only if they have the same signature and the same local invariants.

\begin{definition}
A \emph{quaternion embedding of $\kk$} is a quaternion algebra $B$ together with an embedding of $\Q$-algebras $\kk \to B$.  
\end{definition}

\begin{remark}
The isomorphism class of a quaternion embedding $\kk \to B$ is completely determined by the isomorphism class of the underlying quaternion algebra $B$, as the Noether-Skolem theorem implies that 
 any two embeddings $\kk \to B$ are $B^\times$-conjugate.
\end{remark}

Suppose  $\kk \to B$ is  an indefinite quaternion embedding. The main involution on $B$ restricts to complex conjugation on $\kk$, so the notation $\overline{a}$ for $a\in \co_\kk$ is unambiguous.
Up to $\kk^\times$-scaling, there exists a unique  $j \in B^\times$ satisfying 
$a j = j \overline{a}$ for all $a\in \kk$.  Any such element satisfies $j^2\in \Q^\times$, and  determines a decomposition
\begin{equation}\label{Bdecomp}
B = \kk \oplus \kk j.
\end{equation}
Fix a $\delta \in \kk^\times$ with $\delta^2=\mathrm{disc}(\kk)$.  The   positive involution on $B$ defined  by 
\begin{equation}\label{positive involution}
b^\dagger = \delta \overline{b} \delta^{-1}
\end{equation}
restricts to complex conjugation on $\kk$, and to the identity on $\kk j$.

\begin{lemma}\label{lem:quaternion symplectic}
Up to $\Q^\times$-scaling,
\begin{equation}\label{perfect symplectic}
\lambda(x,y) = \mathrm{Trd}( \delta^{-1} x \overline{y} ) 
\end{equation}
 is the unique symplectic form  $\lambda: B \times B \to \Q$ satisfying 
 \begin{equation}\label{rosati}
 \lambda( b x,y) = \lambda(x,b^\dagger y).
 \end{equation}
  \end{lemma}

\begin{proof}
The proof that \eqref{perfect symplectic}  satisfies \eqref{rosati}  is  elementary, and left to the reader. 
For the uniqueness claim, any  symplectic form $\lambda$ on $B$ satisfying \eqref{rosati}  also satisfies, for all $x,y\in \kk \subset B$, 
\[
\lambda( x j  , y ) = \lambda(  \overline{y} x j ,  1 ) =   -  \lambda(  1,   \overline{y} x j  )
= -\lambda(  1,  x j y     ) =  -\lambda( x j ,y).
\]
Thus $\lambda( x j  , y )=0$, and the two summands in \eqref{Bdecomp} satisfy $\lambda( \kk j , \kk ) =0$.   
This implies that  $\lambda$ is determined by the restrictions  $\lambda|_\kk$ and $\lambda|_{\kk j}$, which are related by 
 \[
 \lambda|_{ \kk j} ( j x, j y)= j^2 \cdot \lambda|_\kk(x,y)  .
 \]
This shows that  $\lambda$ is determined by  $\lambda|_\kk$, which, being a symplectic form on a $2$-dimensional space,  is unique up to scaling.
\end{proof}

It follows from \eqref{rosati} that $B$ admits   a unique $\kk$-hermitian form $\langle-,-\rangle$  satisfying
\begin{equation}\label{symplectic-hermitian}
\lambda(x,y) = \mathrm{Tr}_{\kk/\Q} \langle  \delta^{-1} x,y \rangle  .
\end{equation}
This form is  given by the explicit formula $\langle x,y\rangle = \pi ( x \overline{y})$, where $\pi : B \to \kk$ is projection to the first factor in \eqref{Bdecomp}.

%\begin{remark}
%For any quaternion embedding $\kk \to B$, the  main involution on the quaternion algebra restricts to complex conjugation on $\kk$.  
%\end{remark}

\begin{proposition}\label{prop:unitary to quaternion}
There is a bijection 
\begin{align*}\lefteqn{
\{ \mbox{iso.~classes of indefinite quaternion embeddings of $\kk$}  \}   } \\
& \iso \{  \mbox{iso.~classes of hermitian spaces over $\kk$ of signature $(1,1)$}  \}
\end{align*}
sending $\kk \to B$ to the  hermitian space whose underlying $\kk$-vector space is $W=B$,  endowed with the hermitian form of \eqref{symplectic-hermitian}.  
If $B$ and $W$ are related in this way, then 
\begin{equation}\label{invariant match}
\inv_p(B) = ( -1 , \mathrm{disc}(\kk) )_p \cdot  \inv_p(W) 
\end{equation}
for all places $p\le \infty$, and  there is an isomorphism
 \begin{equation}\label{unitary exceptional}
( \kk^\times \times B^\times) / \Q^\times  \iso \mathrm{GU}(W)  
\end{equation}
that restricts to $\mathrm{ker}( \mathrm{Nrd} : B^\times \to \Q^\times )  \iso \mathrm{SU}(W)$.
%under which the similitude character $\mathrm{GU}(W) \to \Q^\times$ is identified with
%$
%( \alpha,b) \mapsto \alpha \overline{\alpha}/ b\overline{b}.
%$
\end{proposition}

\begin{proof}
Most of this can be found in  \S 5 of \cite{gross04}, so we only add a few comments.  
An easy calculation shows that the hermitian space $W=B$ has determinant
\begin{equation}\label{jsquare}
\det(W) = - j^2 \in \Q^\times / \mathrm{Nm}(\kk^\times),
\end{equation}
and so the $\Q$-basis $\{ 1, \delta , j ,  \delta j \} \subset B$ exhibits an isomorphism 
\[
B \iso  \left(  \frac{ \mathrm{disc}(\kk) , -\det(W)   }{ \Q} \right).
\]
The relation \eqref{invariant match} is clear from this.  
This also shows how to reconstruct the quaternion embedding from the abstract hermitian space $W$, providing the inverse to the function in the statement of the proposition.

The   action of $\kk^\times \times B^\times$ on $W=B$ given by
$
( \alpha , b) \cdot w = \alpha w b^{-1}
$
defines the   isomorphism \eqref{unitary exceptional}.
\end{proof}

\begin{remark}\label{rem:invariant isotropic}
It follows from \eqref{invariant match}  that $\mathrm{inv}_p(B) =1$ if and only if $W\otimes_\Q \Q_p$ is isotropic.
\end{remark}

\begin{proposition}\label{prop:compact dual}
Suppose $\kk \to B$ is an indefinite  quaternion embedding, and let $W=B$ be the signature $(1,1)$ hermitian space of Proposition \ref{prop:unitary to quaternion}.   
   For a $2$-dimensional $\C$-subspace 
   \[
   F^0  W  \subset W_\C = W\otimes_\Q\C,
   \]
    the following are equivalent:
\begin{enumerate}
\item $F^0 W$ is stable under left multiplication by $B$,
\item $F^0 W$ is stable under $\kk$, totally isotropic with respect to \eqref{perfect symplectic}, and free of rank one over $\kk \otimes_\Q \C$.
\end{enumerate}
\end{proposition}

\begin{proof}
First suppose  $F^0W$ satisfies condition (1).
We may choose an isomorphism $B\otimes_\Q \C\iso M_2(\C)$ identifying  $\kk \otimes_\Q \C$ with   the subalgebra of diagonal matrices.  This isomorphism makes  $F^0W$ into a left module over the  $\C$-algebra $M_2(\C)$, whose standard representation is the unique left module of complex dimension $2$.
The standard representation is free of rank one over the subalgebra of diagonal matrices, hence so is  $F^0W$.
To see that $F^0 W$ is totally isotropic under the symplectic form \eqref{perfect symplectic}, note that if we set 
$\kk' = \Q[ j] \subset B$,  the  argument above shows that $F^0W$ is free of rank one over $\kk' \otimes_\Q \C$.
If we choose a generator $x_0 \in F^0W$ then, recalling that $j^\dagger = j$, we find 
\[
\lambda( j x_0 ,x_0) = \lambda(x_0,j^\dagger x_0) =\lambda(x_0, j x_0) = - \lambda( j x_0 ,x_0).
\]
It follows that $\lambda$ is identically $0$ on  $\mathrm{Span}_\C \{ x_0 , j x_0\}= F^0 W$.

Now suppose  $F^0W$ satisfies condition (2), and fix a $\kk\otimes_\Q \C$-module generator $x_0 \in F^0W$.
For any $\alpha \in \kk$ we have
\[
\lambda( \alpha x_0 , j x_0) = \lambda( j \alpha x_0 ,  x_0) 
= \lambda(  \overline{\alpha} j  x_0 ,  x_0) =\lambda(j x_0 , \alpha x_0) = - \lambda( \alpha x_0 , j x_0).
\]
It follows  that $\lambda(  F^0W , j x_0) =0$ and,  as $F^0W$ is a maximal isotropic subspace,
 that $j x_0 \in F^0W$.  Thus  $F^0W$ is stable under left multiplication by $j$, and hence under all of  $B= \kk \oplus \kk j$.
 \end{proof}

%%%%%%%%%%%%%%%%%%%%%%%%%%%%%%%%%%%%

\subsection{Self-dual lattices}

%%%%%%%%%%%%%%%%%%%%%%%%%%%%%%%%%%%%%

We wish to parametrize the set 
\[
\Lat_{(1,1)} = \left\{
\begin{array}{c}
\mbox{isometry classes of self-dual} \\ \mbox{hermitian $\co_\kk$-modules of signature $(1,1)$}
\end{array}
\right\}
\]
where  \emph{hermitian $\co_\kk$-module} means a projective $\co_\kk$-module $L$ of finite rank endowed with an $\co_\kk$-valued hermitian form $\langle - ,- \rangle$, and \emph{self-dual} means that $L$ is equal to its dual lattice
\[
L^\vee = \{ x \in L\otimes \Q : \langle x,L \rangle \subset \co_\kk \}.
\]

There is a natural partition of $\Lat_{(1,1)}$, obtained by declaring $L,L' \in \Lat_{(1,1)}$ to be equivalent if
$L\otimes \Q \iso L'\otimes \Q$ as $\kk$-hermitian spaces.  We express this as a decomposition 
\begin{equation}\label{lattice partition}
 \Lat_{(1,1)} = \bigsqcup_W \Lat_W
\end{equation}
where the disjoint union is over  all  $\kk$-hermitian spaces $W$ of signature $(1,1)$ that admit a self-dual $\co_\kk$-lattice, and
\begin{equation}\label{omegaW}
\Lat_W =  \mathrm{U}(W) \backslash  \{ \mbox{self-dual $\co_\kk$-lattices in $W$} \}.
\end{equation}

Let $\mathrm{CL}(\kk)$ be the ideal class group of $\kk$.
The \emph{principal genus} 
$
\mathrm{CL}_0(\kk) \subset \mathrm{CL}(\kk)
$
 is characterized either as   the subgroup of squares, or by
 \[
\mathrm{CL}_0(\kk) \iso   \kk^\times \backslash  \{ \mbox{fractional ideals }\mathfrak{a} \subset \kk  :   \mathrm{Nm} (\mathfrak{a} )  \in \mathrm{Nm}(\kk^\times)  \} .
 \] 
 More adelically, if we let $\kk^1$ denote  the kernel of the norm map $\kk^\times \to \Q^\times$, and similarly for $\widehat{\kk}^1$ and  $\widehat{\co}_\kk^1$,  then  
 \[
\mathrm{CL}_0(\kk)  \iso \kk^1 \backslash \widehat{\kk}^1  / \widehat{\co}_\kk^1 .
\]
Using this last isomorphism and the isomorphism $\kk^\times/ \Q^\times \iso \kk^1$ of Hilbert's Theorem 90, one can prove  the classical formula (originating in  Gauss's genus theory of binary quadratic forms)
\begin{equation}\label{genus theory}
  | \mathrm{CL}_0(\kk) |   =  \frac{ | \mathrm{CL}(\kk) | }{ 2^{o(\kk)-1} },
\end{equation}
where $o(\kk)$  is the number of prime divisors of $\mathrm{disc}(\kk)$.

\begin{proposition}\label{prop:steinitz orbits}
Assume that $\mathrm{disc}(\kk)$ is odd.
The  Steinitz class map
\begin{equation}\label{steinitz}
\mathrm{St} : \Lat_{(1,1)} \to \mathrm{CL}(\kk) ,
\end{equation}
which sends $L \in \Lat_{(1,1)}$ to  the  rank $1$ projective $\co_\kk$-module 
$\mathrm{St}(L) = \bigwedge_{\co_\kk}^2  L$,  is  bijective.  
Under this bijection,  the image of each subset \eqref{omegaW}  is a  coset of the principal genus,  and in fact 
\[
\mathrm{St}(\Lat_W) = \mathfrak{a} \mathrm{CL}_0(\kk)
\]
for any fractional ideal  $\mathfrak{a}$ with  $\mathrm{Nm} (\mathfrak{a} ) = - \det(W)$ in  $\Q^\times/\mathrm{Nm}(\kk^\times)$.
\end{proposition}

\begin{proof}
As in Lemma 2.11(i) of \cite{KRunitaryII}, 
there are $2^{ o(\kk) -1}$ distinct hermitian spaces $W$ that contribute to the disjoint union \eqref{lattice partition}.
Fix one.  By a theorem of Jacobowitz\footnote{and our assumption that  $\mathrm{disc}(\kk)$ odd} \cite{jacobowitz}, any two self-dual lattices in $W$ are isometric everywhere locally. Fixing one such lattice $L \subset W$, we obtain a bijection
\[
\mathrm{U}(W) \backslash \mathrm{U}( W_{\widehat{\Q}} ) / \mathrm{U}(  L_{\widehat{\Z}}   ) \map{ g \mapsto gL} \Lat_W.
\]
Using  strong approximation for the (simply connected) group $\mathrm{SU}(W)$, one sees that the function
\[
\mathrm{U}(W) \backslash \mathrm{U}( W_{\widehat{\Q}} ) / \mathrm{U}(  L_{\widehat{\Z}}   ) 
\map{\det} \kk^1 \backslash \widehat{\kk}^1  / \widehat{\co}_\kk^1 \iso \mathrm{CL}_0(\kk)
\]
is also bijective, and so 
\begin{equation}\label{self-dual count}
| \Lat_W | =  | \mathrm{CL}_0(\kk) | .
\end{equation}

The above discussion,  \eqref{lattice partition}, and  \eqref{genus theory}  show that 
$|\Lat_{(1,1)} | = | \mathrm{CL}(\kk) |$, and so to prove the bijectivity of 
 \eqref{steinitz} it suffices to construct a section to the function.  
This can be done by  sending a fractional ideal $\mathfrak{a} \subset \kk$ to the 
self-dual hermitian  $\co_\kk$-module $L  = \co_\kk \oplus \mathfrak{a}$, where the first factor in the orthogonal direct sum is endowed with the hermitian form $\langle x,y\rangle = x\overline{y}$, and the second with 
$\langle x,y\rangle = - x\overline{y} / \mathrm{Nm}(\mathfrak{a})$.

For any $L \in \Lat_{(1,1)}$ and any fractional ideal $\mathfrak{c} \subset \kk$,  we endow the $\co_\kk$-module 
$
L^\mathfrak{c} = \mathfrak{c} \otimes_{\co_\kk}  L   
$
with the  hermitian form
\[
\langle c_1 \otimes x_1 , c_2 \otimes x_2 \rangle = \frac{ c_1\overline{c_2} }{ \mathrm{Nm}(\mathfrak{c}) }
 \cdot  \langle x_1, x_2\rangle.
\]
This defines an action of $\mathrm{CL}(\kk)$  on $\Lat_{(1,1)}$, satisfying
$ \mathrm{St}(  L^\mathfrak{c})  = \mathfrak{c}^2  \mathrm{St}(L)$. 
The determinants of   $L \otimes \Q$ and $L^\mathfrak{c} \otimes \Q$   are equal in $\Q^\times / \mathrm{Nm}(\kk^\times)$, and so comparison of  local invariants shows they   are isomorphic as $\kk$-hermitian spaces.   This proves that each subset $\Lat_W \subset \Lat_{(1,1)}$ is stable under  the action.

If we now use \eqref{steinitz} to identify $\Lat_{(1,1)} = \mathrm{CL}(\kk)$ as $\mathrm{CL}(\kk)$-sets, but with the action on the right hand side given by  $\mathfrak{c} * \mathfrak{a} = \mathfrak{c}^2 \mathfrak{a}$, then each $\Lat_W$ is a union of orbits.  The orbits are exactly the cosets of the principal genus, and the counting formula \eqref{self-dual count} shows that each $\Lat_W$ is a single coset.
The final claim, specifying which coset it is, follows from the explicit  section to \eqref{steinitz} constructed above. 
\end{proof}

\begin{remark}\label{rem:lattice splitting}
As a corollary of the proof, every $L \in \Lat_{(1,1)}$ can be split as an orthogonal direct sum 
$L = \co_\kk \oplus \mathfrak{a}$.  This is false if we drop the assumption that $\mathrm{disc}(\kk)$ is odd.
\end{remark}

%%%%%%%%%%%%%%%%%%%%%%%%%%%%%%%%%%%%

\subsection{Construction of Eichler orders}

%%%%%%%%%%%%%%%%%%%%%%%%%%%%%%%%%%%%

Fix a $\kk$-hermitian space $W$ of signature $(1,1)$.
Let  $\kk \to B$ be the indefinite quaternion embedding  corresponding to $W$ under the bijection of Proposition \ref{prop:unitary to quaternion}.

\begin{proposition}\label{prop:eichler construction}
Assume that $\mathrm{disc}(\kk)$ is odd, and suppose $L\subset W$ is a self-dual $\co_\kk$-lattice.
 There exists an  Eichler order $R \subset B$ of level
 \begin{equation}\label{level}
N = - \mathrm{disc}(\kk) / \mathrm{disc}(B) 
\end{equation}
(in particular, this is an integer) containing $\co_\kk$ such that:
 \begin{enumerate}
  \item
 The positive involution \eqref{positive involution} satisfies  $R^\dagger =R$.
 \item
Under the symplectic form \eqref{perfect symplectic}, $R$ is a self-dual $\Z$-lattice.
   \item
Under the $\kk$-hermitian form  of \eqref{symplectic-hermitian}, $R$ is a self-dual $\co_\kk$-lattice isometric to $L$.
 \end{enumerate}
\end{proposition}

\begin{proof}
Write $B=\kk \oplus \kk j$ be as in \eqref{Bdecomp}, and  fix a fractional ideal 
$\mathfrak{a} \subset \kk$ representing the Steinitz class $\mathrm{St}(L)$.
Combining \eqref{jsquare} with the final claim of Proposition \ref{prop:steinitz orbits} shows that 
\[
\mathrm{Nm}(\mathfrak{a})  \cdot j^2 \in \mathrm{Nm}(\kk^\times),
\]
and so we may  rescale $\mathfrak{a}$ by an element of $\kk^\times$ to assume that 
$
\mathrm{Nm}(\mathfrak{a})  \cdot j^2 =1.
$
This guarantees that 
\[
R =  \co_\kk \oplus \mathfrak{a} j
\]
is an order in $B$, and direct calculation shows that it has reduced discriminant $-\mathrm{disc}(\kk)$. 
Note that the  reduced discriminant is squarefree,  by the assumption that $\mathrm{disc}(\kk)$ is odd.
 The equality $R^\dagger =R$ is clear, as $\dagger$ restricts to complex conjugation on  $\kk$, and to the identity on $\kk j$. 

The reduced discriminant of an order is a multiple of the reduced discriminant of any order that contains it, and the reduced discriminant of any maximal order is $\mathrm{disc}(B)$.   Hence \eqref{level} is an integer.

We claim that \emph{any} order $R\subset B$ of squarefree reduced discriminant $N \mathrm{disc}(B) $   is an Eichler order (necessarily of level $N$).  This   can be checked locally. 
At  a prime  $p\nmid N$, $R_p$ is maximal, and there is nothing to check.
  At a prime $p\mid N$, we may identify $B_p \iso M_2(\Q_p)$ in such a way that $R_p \subset M_2(\Z_p)$ with index $p$.  Let $A \subset M_2(\F_p)$ be the image of $R_p$ under the reduction map.  
  It is a $3$-dimensional $\F_p$-algebra whose  Jacobson radical $J(A)$ is nilpotent and  nonzero; indeed, if $J(A)=0$ then $A$ would be a  semisimple $\F_p$-algebra of dimension $3$, and no such thing embeds into $M_2(\F_p)$. 
   It follows that $J(A) \cdot \F_p^2 \subset \F_p^2$ is an $A$-stable  line, and so $A$ is conjugate to the algebra of upper triangular matrices.  From this one can see that  the isomorphism $B_p\iso M_2(\Q_p)$ may be chosen so that 
\[
R_p=\left\{    \left( \begin{matrix} a & b \\ pc & d  \end{matrix}  \right) : a,b,c,d \in \Z_p\right\} \subset M_2(\Z_p),
\]
which is an Eichler order of level $N$.

The self-duality of $R$ under  the hermitian form of  \eqref{symplectic-hermitian} is a direct calculation.
 It is isometric to $L$ because both have Steinitz class $\mathfrak{a}$, which determines the isometry class by Proposition \ref{prop:steinitz orbits}.  
 Self-duality under the symplectic form follows from self-duality under the hermitian form and the relation \eqref{symplectic-hermitian}.
\end{proof}

%%%%%%%%%%%%%%%%%%%%%%%%%%%%%%%%%%%%

\section{Unitary Shimura curves}
\label{s:unitary curves}

%%%%%%%%%%%%%%%%%%%%%%%%%%%%%%%%%%%%%

Let $\kk$ be a quadratic imaginary field of odd discriminant $\mathrm{disc}(\kk)$.

%%%%%%%%%%%%%%%%%%%%%%%%%%%%%%%%%%%%

\subsection{The moduli problem}

%%%%%%%%%%%%%%%%%%%%%%%%%%%%%%%%%%%%%

As with quaternionic Shimura curves, we will define unitary Shimura curves as moduli spaces of  abelian surfaces with additional structure.

\begin{definition}
 An  \emph{$\co_\kk$-abelian surface of signature $(1,1)$} over a scheme $S$ is an abelian scheme $A \to S$ of relative dimension two, together with a ring homomorphism $\co_\kk \to \End(A)$ such that every $a\in \co_\kk$
acts on $\Lie(A)$ with characteristic polynomial
\[
x^2-(a+\overline{a}) x + a\overline{a} \in \Z[x],
\]
viewed as an element of $\co_S[x]$. 
 \end{definition}

Let $\mathcal{X}_{(1,1)}$ be the Deligne-Mumford stack whose functor of points assigns to a $\Z$-scheme 
$S$ the groupoid of pairs $(A,\lambda)$ in which  $A$ is an $\co_\kk$-abelian surface  of signature $(1,1)$ 
over  $S$, and  $\lambda : A \to A^\vee$ is a principal polarization such that 
\[
\overline{a} = \lambda^{-1}  \circ a^\vee  \circ\lambda \in \End(A)
\]
for every $a\in \co_\kk$; in other words, the Rosati involution on $\End(A)$ restricts to complex conjugation on $\co_\kk$.  

  It is known that $\mathcal{X}_{(1,1)}$ is a regular Deligne-Mumford stack of dimension $2$, flat and of finite type over $\Z$, and smooth over $\Z[1/\mathrm{disc}(\kk)]$. 
Note that regularity and flatness use the  local model calculations of Pappas \cite{pappas}, which requires our standing hypothesis that $\mathrm{disc}(\kk)$ is odd.

By  Proposition 2.12(i) of \cite{KRunitaryII} and its proof, there is a decomposition 
\[
\mathcal{X}_{(1,1)} = \bigsqcup_W \mathcal{X}_W
\]
into open and closed substacks in which,
recalling Proposition \ref{prop:steinitz orbits},  the disjoint union is over the $2^{o(\kk) - 1 } $ isometry classes of $\kk$-hermitian spaces $W$ of signature $(1,1)$ that  contain a self-dual $\co_\kk$-lattice.
The  substack $\mathcal{X}_W$ is characterized as follows:
for any geometric point  $x \to \mathcal{X}_W$ and any prime $\ell \neq \mathrm{char}(x)$,   there exist a $\kk_\ell$-linear isomorphism
\[
\mathrm{Ta}_\ell(A_x) \otimes  \Q_\ell \iso W \otimes \Q_\ell
\]
and a $\Q_\ell$-linear isomorphism $\Q_\ell(1) \iso \Q_\ell$  that identify  the $\ell$-adic Weil pairing on the left (induced by the principal polarization on $A_x$) with the symplectic form 
$
(s,t) \mapsto  \mathrm{Tr}_{\kk_\ell/\Q_\ell}   \langle  \delta^{-1} s , t \rangle  
$
on the right.  Here  $\langle -,-\rangle$ is the hermitian form on $W$,  and $\delta \in \kk^\times$ is any element with $\delta^2 = \mathrm{disc}(\kk)$.

 \begin{remark}
The stacks $\mathcal{X}_W$ are  typically disconnected.
The above decomposition  of $\mathcal{X}_{(1,1)}$ will be refined in Theorem \ref{thm:unitary volume} below.
\end{remark}

%%%%%%%%%%%%%%%%%%%%%%%%%%%%%%%%%%%%

\subsection{A morphism of Shimura curves}
\label{ss:the morphism}

%%%%%%%%%%%%%%%%%%%%%%%%%%%%%%%%%%%%%

Fix an indefinite  quaternion embedding $\kk \to B$, let 
$
W=B
$
 be the signature $(1,1)$ hermitian space of Proposition \ref{prop:unitary to quaternion}, and suppose $L \subset W$ is a self-dual $\co_\kk$-lattice.
Fix a level $N$ Eichler order $R \subset B$ as in  Proposition \ref{prop:eichler construction}, and a maximal order $\co_B$ containing it.
After replacing $L$ by an isometric lattice in $W$, we may assume  $L=R$.

In this subsection we will  construct a morphism 
\begin{equation}\label{generic key}
q_L : \mathcal{X}_B(N)_{/\Q} \to \mathcal{X}_{W/\Q}
\end{equation}
of $\Q$-stacks by realizing the source and target as canonical models of Shimura varieties.   
We emphasize  that the map depends on the  isometry class of  $L\subset W$.  
Indeed, the domain is connected but the codomain usually is not.
As $L$ varies, the  maps   will take values in different connected components.

Define a chain of $\Z$-lattices
\[
\Lambda_0 \subset \Lambda_R \subset \Lambda_1
\]
 in $B$ as follows.    First set 
$\Lambda_1 =\co_B$ and $\Lambda_R = R.$
The subgroup 
\begin{equation}\label{eichler subgroup}
\Lambda_R /N \Lambda_1 \subset \Lambda_1 /N \Lambda_1
\end{equation}
contains a unique index $N$ subgroup stable under the left action of $\co_B$. Namely, identify $\Lambda_R /N \Lambda_1$
 with the subgroup of upper triangular matrices in 
 \[
\Lambda_1 /N \Lambda_1 \iso  \co_B/N\co_B\iso M_2(\Z/N\Z) ,
 \]
  and then take the subgroup of  matrices whose first column is $0$.    
Thus there is a unique left  $\co_B$-stable  sublattice $\Lambda_0 \subset \Lambda_R$  of  index $N$ that contains $
N \Lambda_1$.

Denote by $\mathcal{H}$ the hermitian symmetric domain whose points  $z \in \mathcal{H}$ parametrize  weight $-1$ Hodge structures 
  \[
W_\C = W_z^{0,-1} \oplus  W_z^{-1,0} 
\]
whose  filtrations $F^0_z W = W_z^{0,-1} $ satisfy the equivalent conditions of Proposition \ref{prop:compact dual}.
One can easily verify that $\mathcal{H}$  is isomorphic to the union of the upper and lower complex half-planes.

The reductive group  $G_0 = B^\times$   acts on $W=B$ by \emph{right} multiplication 
\[
g \cdot w = wg^{-1},
\]
and this induces a transitive action of $G_0(\R)$ on $\mathcal{H}$.  
If we define a compact open subgroup   $K_0 = \widehat{R}^\times$ of $G_0(\A_f)$, there is an isomorphism
\[
G_0(\Q) \backslash (  \mathcal{H} \times G_0(\A_f)   / K_0 )  \iso \mathcal{X}_B(N)(\C) 
\]
sending the double coset of a pair $(z,g)$ to the isogeny $f : A_0 \to A_1$ defined by the   $\co_B$-abelian surfaces 
\begin{equation}\label{complex universal}
A_i(\C)   = F^0_z W  \backslash W_\C /  g \cdot \Lambda_i 
\end{equation}
and the inclusion $ g\cdot \Lambda_0  \subset g \cdot \Lambda_1 $.

The group  $G=\mathrm{GU}(W)$ of unitary similitudes of $W$ also acts transitively on $\mathcal{H}$, and the pair
$(G,\mathcal{H})$ is a Shimura datum with reflex field $\Q$.
If we let  $K \subset G(\A_f)$ be the stabilizer of the lattice $L=\Lambda_R$, 
there is an isomorphism\footnote{The surjectivity requires the theorem of Jacobowitz cited in the proof of Proposition \ref{prop:steinitz orbits}.  Jacobowitz's result, which requires our assumption that $\mathrm{disc}(\kk)$ is odd,  guarantees that $G(\A_f)$ acts transitively on set of  self-dual $\co_\kk$-lattices in $W$.}    of complex orbifolds
\[
G(\Q)  \backslash (  \mathcal{H} \times G(\A_f)  / K  )  \iso \mathcal{X}_W(\C) 
\]
sending the double coset of a pair $(z,g)$ to the   $\co_\kk$-abelian surface
\begin{equation}\label{complex intermediate}
A_R(\C)   = F^0_z W  \backslash W_\C /   g \cdot \Lambda_R 
\end{equation}
of signature $(1,1)$.
 To polarize this abelian surface, let $\nu(g) \in \widehat{\Q}^\times$ be the similitude factor of $g$, and write
\[
\nu(g) = r \cdot u
\]
 with $r\in \Q^\times$ and $u\in \widehat{\Z}^\times$.  By Proposition \ref{prop:eichler construction}, the $\Z$-lattice $g \cdot \Lambda_R$ is self-dual  under the rescaled symplectic form $r^{-1} \lambda$, which  defines (plus or minus) a principal polarization.

Proposition \ref{prop:unitary to quaternion} (and its proof) imply that  $G_0 \subset G$
as subgroups of $\GL_\kk(W)$, and the pair $(G_0,\mathcal{H})$ is again a Shimura datum with reflex field $\Q$.
  It is clear from the constructions that $K_0 \subset K$, as right multiplication by $R$ stabilizes $\Lambda_R=R$.

\begin{proposition}\label{prop:complex degree}
The morphism  \eqref{generic key} induced by the morphism of Shimura data
$
(G_0,\mathcal{H}) \to (G,\mathcal{H})
$
is finite \'etale of degree $|\co_\kk^\times|$ over its image (as we have noted already, it need not be surjective). 
\end{proposition}

\begin{proof}
 Fix a  connected component $\mathcal{H}^+ \subset \mathcal{H}$.
The domain of the orbifold morphism  
\[
q_L:\mathcal{X}_B(N)(\C) \to \mathcal{X}_W(\C)
\]
 is connected, and  if we replace the codomain by the  image, the morphism takes the form 
\[
q_L :  R^1  \backslash \mathcal{H}^+ \to \mathrm{U}(L) \backslash \mathcal{H}^+,
\]
where $R^1 \subset R^\times$ is the kernel of the reduced norm $R^\times \to \{\pm 1\}$.
The exactness of 
\[
1 \to \mathrm{SU}(L)  \to \mathrm{U}( L ) \map{\det} \co_\kk^\times \to 1
\]
 follows easily from Remark \ref{rem:lattice splitting}, while  the final claim of Proposition \ref{prop:unitary to quaternion} identifies  $R^1\iso \mathrm{SU}(L)$.  Thus 
$
\deg(q_L) = |  \mathrm{U}(L) :  R^1 ] = | \co_\kk^\times |.
$
\end{proof}

%%%%%%%%%%%%%%%%%%%%%%%%%%%%%%%%%%%%

\subsection{Extension to integral models}

%%%%%%%%%%%%%%%%%%%%%%%%%%%%%%%%%%%%%

The goal of this subsection is show that the morphism   of $\Q$-stacks  \eqref{generic key}  admits a unique extension to a morphism $q_L : \mathcal{X}_B(N) \to \mathcal{X}_W$  of integral models.

As in Definition \ref{def:intermediate}, the Eichler order $R\subset B$ determines a factorization 
\[
A_0 \to A_R \to A_1
\]
of the universal degree $N^2$ isogeny $f:A_0 \to A_1$ over $\mathcal{X}_B(N)$,  in which the intermediate abelian surface  comes with an action  $R \subset \End(A_R)$.  
Restricting this action to  $\co_\kk\subset R$ makes  
\[
A_R \to \mathcal{X}_B(N)
\]
 into an $\co_\kk$-abelian surface of signature $(1,1)$.  Indeed, as $\mathcal{X}_B(N)$ is flat over $\Z$, it is enough to verify the signature condition over the generic fiber, where the natural morphism $\Lie(A_0) \to \Lie(A_R)$ restricts to an isomorphism.  Thus the signature condition on $\Lie(A_R)$ follows from Drinfeld's determinant condition imposed on $\Lie(A_0)$ in \S \ref{ss:drinfeld moduli}.

\begin{proposition}\label{prop:the polarization}
The intermediate abelian surface admits a unique principal polarization $\lambda_R :A_R \to A_R^\vee$ such that
 the induced Rosati involution  restricts to the positive involution \eqref{positive involution} on $R \subset \End(A_R)$.
\end{proposition}

\begin{proof}
Standard arguments (see Chapters 6 and 7 of \cite{hida}, for example), show that 
the functor  that assigns to a scheme $S$ the groupoid 
\[
\mathcal{X}^\dagger_B(N)(S) = \left\{  \begin{array}{c}
\mbox{morphisms $S \to \mathcal{X}_B(N)$ together with a} \\
\mbox{principal polarization }  A_{R/S} \to A_{R/S}^\vee \\
\mbox{whose Rosati involution restricts to $\dagger$ on $R$}
\end{array}  \right\}
\]
is represented by a proper and unramified  morphism  of Deligne-Mumford stacks  
\[
g : \mathcal{X}^\dagger_B(N) \to \mathcal{X}_B(N).
\]

The morphism  $g$ is injective on geometric points of any characteristic.
Indeed, given a geometric point  $x \to \mathcal{X}_B(N)$, choose a prime $\ell \neq \mathrm{char}(x)$.
A lift of $x$ to $\mathcal{X}^\dagger_B(N)$  is determined by a principal polarization  $ \lambda_x : A_{R,x} \to A^\vee_{R,x}$ whose induced $\ell$-adic Weil pairing 
\[
\lambda_x  : \mathrm{Ta}_\ell( A_{R,x} ) \times  \mathrm{Ta}_\ell( A_{R,x} ) \to \Z_\ell
\]
satisfies $\lambda_x (r s, t) =  \lambda_x( s, r^\dagger t)$ for all $r\in R$. 
As $\mathrm{Ta}_\ell( A_{R,x} ) \otimes\Q_\ell$ is free of rank one over $B_\ell$, it follows from  Lemma \ref{lem:quaternion symplectic} (or rather, its  $\ell$-adic analogue) that any two such polarizations span the same $\Q_\ell$-line in $\Hom( A_{R,x} ,  A^\vee_{R,x}) \otimes \Q_\ell$, from which it follows that they are equal.

The morphism  $g$ is surjective on complex points.  Indeed, at any complex point of $\mathcal{X}_B(N)$ the universal $f : A_0 \to A_1$ has the form \eqref{complex universal} for some 
\[
(z,g) \in \mathcal{H} \times G_0(\A_f) \subset \mathcal{H} \times G(\A_f), 
\]
and the intermediate abelian surface takes the form \eqref{complex intermediate} for the same pair.
  We have already explained how to construct a polarization of the desired type on \eqref{complex intermediate}.

The morphism $g$  is quasi-finite and proper, so finite by \cite[\href{https://stacks.math.columbia.edu/tag/02OG}{Tag 02OG}]{stacks-project}.
Being finite, unramified, and injective on geometric points,    
it becomes a closed immersion after pullback to an \'etale cover of $\mathcal{X}_B(N)$, by  
\cite[\href{https://stacks.math.columbia.edu/tag/04HG}{Tag 04HG}]{stacks-project}.  
By \cite[\href{https://stacks.math.columbia.edu/tag/02L6}{Tag 02L6}]{stacks-project}  it was already a closed immersion before pullback to the \'etale cover.

At this point we know that $g$ is a closed immersion, bijective on complex points, whose codomain is integral and flat of finite type over $\Z$.  \'Etale locally, such a morphism has the form
$
\Spec(D/I) \to \Spec(D),
$
where  $D$ is an integral domain flat and of finite type over $\Z$ and $I \subset D$ is an ideal contained in the kernel of any ring homomorphism $D\to \C$.  But $D$ admits an injective homomorphism to $\C$, and so $I=0$.
It follows that   $g$ is an isomorphism.
\end{proof}

%\begin{remark}
%Proposition \ref{prop:the polarization} is similar to a result of Drinfeld  (see Proposition III.3.5 of \cite{boutot-carayol}) asserting the existence of  a unique principal polarization on $A_0$ compatible with a \emph{different} positive involution on $B$, defined as in  \eqref{positive involution} but with our $\delta \in \kk^\times$ replaced by an element of  $B^\times$ whose square is $\mathrm{disc}(B)$.
%Unless we are in the special case where $\mathrm{disc}(B)=\mathrm{disc}(\kk)$, 
%the  involution used by Drinfeld  will not preserve the subalgebra $\kk \subset B$, and so the induced polarization on $A_0$ is not compatible with the action $\co_\kk \subset \End(A_0)$.
%
%Our proof of Proposition \ref{prop:the polarization} is more direct than that of Drinfeld, and with good reason: Drinfeld must construct polarizations on $\co_B$-abelian surfaces in order to prove the existence of the moduli space $\mathcal{X}_B(N)$, whereas we have the benefit of already knowing its existence and its  local  structure.
%\end{remark}
%

\begin{proposition}\label{prop:key morphism}
There is a unique morphism 
\[
q_L : \mathcal{X}_B(N) \to \mathcal{X}_W
\]
extending the morphism  \eqref{generic key} already constructed in the generic fiber. 
It is relatively representable and  finite \'etale of degree $| \co_\kk^\times|$ over its image, 
and the pullback via $q_L$ of the universal $\co_\kk$-abelian surface $A \to \mathcal{X}_W$  is  isomorphic to the intermediate abelian surface  $A_R \to \mathcal{X}_B(N)$ of Definition \ref{def:intermediate}.
\end{proposition}

 \begin{proof}
By endowing the intermediate abelian surface $A_R \to \mathcal{X}_B(N)$ with its action of $\co_\kk \subset R$ and  the polarization of Proposition \ref{prop:the polarization}, we obtain a  morphism
\[
q_L : \mathcal{X}_B(N) \to \mathcal{X}_{(1,1)}.
\]
 It is clear from the explicit constructions of  \S \ref{ss:the morphism} that this  extends the morphism  \eqref{generic key} already constructed in the complex fiber.
The flatness of the integral models  over $\Z$  implies both that $q_L$ takes values in $\mathcal{X}_W$ (as this is true in the generic fiber) and that the extension of \eqref{generic key} to integral models is unique.

It remains to prove that $q_L$ is relatively representable and finite \'etale.  
To do so,  factor $q_L$ as a composition
\begin{equation}\label{first factor}
\mathcal{X}_B(N) \to    \mathcal{X}_W' \to \mathcal{X}_W,
\end{equation}
where  the Deligne-Mumford stack in the middle is defined as follows: for any scheme $S$, 
an object of  the groupoid $\mathcal{X}_W'(S)$  is a  principally polarized $\co_\kk$-abelian surface  $(A,\lambda) \in \mathcal{X}_W(S)$   together with an extension of the $\co_\kk$-action to an action
$
R \subset \End(A)
$
such that the Rosati involution induced by $\lambda$ restricts to $\dagger$ on $R$.
Standard arguments (again, see Chapters 6 and 7 of \cite{hida}) show that the forgetful morphism
$\mathcal{X}'_W \to \mathcal{X}_W$ is relatively representable, finite, and unramified. 

Of course, the first arrow in \eqref{first factor} sends $f:A_0 \to A_1$ to its intermediate abelian surface, endowed  with the polarization of Proposition \ref{prop:the polarization}.

\begin{lemma}
The first arrow in \eqref{first factor} is a  closed immersion.  In particular,  it is relatively representable, finite,  and unramified.
\end{lemma}

\begin{proof}
Set $\Z_N=\prod_{p\mid N} \Z_p$, fix an isomorphism 
$
\co_B \otimes_\Z \Z_N \iso M_2(\Z_N)
$
 identifying 
\[
R\otimes_\Z \Z_N = \left\{  \left(\begin{matrix} a & b \\ N c & d \end{matrix}\right) : a,b,c,d \in \Z_N \right\},
\]
and define elements of $R\otimes_\Z \Z_N$ by
\[
e = \left( \begin{matrix} 1 & 0 \\ 0 & 0 \end{matrix}\right),
\qquad 
e' = \left( \begin{matrix} 0 & 0 \\ 0 & 1 \end{matrix}\right) ,
\qquad
w= \left(\begin{matrix} 0 & 1 \\ N & 0 \end{matrix}\right).
\]

Suppose we have a scheme $S$ and  object  $f:A_0 \to A_1$ of  $\mathcal{X}_B(N)(S)$.
If we define finite flat  $R$-stable subgroup schemes 
\begin{equation}\label{PandQ}
P  = e'  \ker(  w : A_R[N] \to A_R[N] )  ,\qquad 
Q = e A_R[N] \oplus P
\end{equation}
of $A_R$ then, recalling the subgroup scheme $C_R = e  \ker(f) \subset A_0$ used in Definition \ref{def:intermediate}, we have
\[
 P= \ker(f) / C_R \subset  A_0/C_R = A_R,
 \]
  and the induced morphism
 \[
 A_1 \iso A_0/\ker(f)  \iso A_R/P \to A_R/ Q
 \]
has the same kernel as the dual isogeny $f^\vee :A_1 \to A_0$ of Definition \ref{def:AL}.
Indeed,  this is easily verified  if we are in the situation $\co_B\iso M_2(\Z)$ of Remark \ref{rem:split intermediate}, so that $f:A_0 \to A_1$ has the form \eqref{elliptic isogeny};
the general case is no different, as one can check the claims after replacing $f: A_0  \to A_1$ with the induced morphisms of $p$-divisible groups for all $p\mid N$,  which again have the form \eqref{elliptic isogeny}.

The above paragraph shows that  we can recover $f:A_0 \to A_1$ from $A_R$ (with its $R$-action), as the dual isogeny to $A_R/P \to A_R/ Q$.
In other words,   for any scheme $S$ the first arrow in \eqref{first factor} defines a fully faithful functor 
\[
\mathcal{X}_B(N)(S) \to \mathcal{X}_W'(S),
\]
whose inverse over the essential image is explicitly known. We must show the essential image is defined by the inclusion of a closed substack of $\mathcal{X}_W'$.

Consider the universal object $A \to \mathcal{X}_W'$.  
There are  finite flat $R$-stable subgroup schemes  $P\subset Q$ of $A$  defined exactly as in \eqref{PandQ}, but with $A_R$ replaced by $A$.
 Imposing the conditions 
 \[
 \mathrm{rank}(P) = N \qquad\mbox{and}\qquad \mathrm{rank}(Q)=N^3
 \]
 cuts out an open and closed substack of $\mathcal{X}_W'$, and imposing the condition that the actions of $R$ on 
\[
A_1=A/P  \qquad  \mbox{and} \qquad A_0=A/Q
\]
 extend to $\co_B$-actions satisfying the determinant condition of Definition \ref{def:OBsurface} cuts out a closed substack.
 Thus there is a maximal closed substack  over which all of these conditions are satisfied.
 
 From what was said above, the first arrow in \eqref{first factor} factors through this closed substack, and over this closed substack the arrow has an inverse,   defined by  dualizing  the $\co_B$-linear isogeny  $A_1 \to A_0$ determined by the inclusion $P\subset Q$.
 \end{proof}

We can now complete the proof of Proposition \ref{prop:key morphism}.  
The morphism $q_L : \mathcal{X}_B(N) \to \mathcal{X}_W$ is relatively representably, finite, and unramified, as each of the arrows in the composition \eqref{first factor} has these properties.
In particular, by \cite[\href{https://stacks.math.columbia.edu/tag/04HG}{Tag 04HG}]{stacks-project}, the induced maps on \'etale local rings are surjective.   
As the source and target of $q_L$ are regular of dimension $2$, their \'etale local rings are integral domains of the same dimension, and  Krull's Hauptidealsatz implies that any surjection between them is an  isomorphisms.  
Hence $q_L$ induces isomorphisms on \'etale local rings,  so is \'etale.  
As the source and target of $q_L$ are flat over $\Z$, the calculation of its degree can be done in characteristic $0$.  This was done in Proposition \ref{prop:complex degree}.   
\end{proof}

%%%%%%%%%%%%%%%%%%%%%%%%%%%%%%%%%%%

\subsection{Connected components and level structure}
\label{ss:level}

%%%%%%%%%%%%%%%%%%%%%%%%%%%%%%%%%%%

Recall from \eqref{omegaW} that $\Lat_W$ denotes the set of isometry classes of self-dual $\co_\kk$-lattices in $W$.
  It follows from the discussion of \S \ref{ss:the morphism} that  
\begin{equation}\label{complex components}
\mathcal{X}_W (\C) = \bigsqcup_{ L \in \Lat_W }  U(L) \backslash \mathcal{H}^+,
\end{equation}
where  each $U(L) \backslash \mathcal{H}^+$ is the orbifold quotient of a connected domain $\mathcal{H}^+$ (isomorphic to the complex upper half-plane) by the action of a discrete group.   Compare with Proposition 3.1 of \cite{KRunitaryII}.

\begin{proposition}\label{prop:components}
There is a unique decomposition
\[
\mathcal{X}_W = \bigsqcup_{ L \in \Lat_W} \mathcal{C}_L 
\]
into a disjoint union of  open and closed substacks, such that taking complex points recovers 
\eqref{complex components}.   For each $L \in \Lat_W$ the morphism 
$
\mathcal{C}_L \to \Spec(\Z)
$
is flat with geometrically connected fibers, and the morphism $q_L$ of Proposition \ref{prop:key morphism} has image  $\mathcal{C}_L$.
\end{proposition}

\begin{proof}
For each $L\in \Lat_W$ we have constructed in Proposition \ref{prop:key morphism} a  morphism
$
q_L : \mathcal{X}_B(N) \to \mathcal{X}_W.
$
As this morphism is finite \'etale,  its image $\mathcal{C}_L \subset \mathcal{X}_W$ is both open and closed.
As $\mathcal{X}_B(N)$ has geometrically connected fibers, so does $\mathcal{C}_L$.  The morphism $q_L$ was constructed so that its image in the complex fiber is the component $U(L) \backslash \mathcal{H}^+$ of \eqref{complex components}, and so we have found a connected component  satisfying 
$
\mathcal{C}_L(\C) \iso U(L) \backslash \mathcal{H}^+ .
$
The existence and uniqueness of the decomposition follow from this and the flatness of  $\mathcal{X}_W$ over $\Z$.
\end{proof}

Fix a self-dual $\co_\kk$-lattice $L\subset W$.
The hermitian form on $W$ determines an alternating $\Q$-bilinear  form $\lambda : W\times W \to \Q$, as in \eqref{symplectic-hermitian}, under which $L$ is again self-dual. 
Hence $\lambda$ induces a perfect alternating form
\begin{equation}\label{fixed weil}
\lambda_m : L / mL \times L/mL \to \Z/m\Z
\end{equation}
for any integer $m\ge 1$.

Let $A\to \mathcal{C}_L$ be the restriction of the universal polarized $\co_\kk$-abelian surface over $\mathcal{X}_W$ to the connected component determined by $L$.  Denote by 
\[
\mathcal{C}_L(m) \to \Spec(\Z[1/m])
\]
the Deligne-Mumford stack whose functor of points assigns to a $\Z[1/m]$-scheme $S$ the groupoid 
whose objects consist of a morphism $S \to \mathcal{C}_L$ together with  $\co_\kk$-linear  isomorphisms of group schemes 
\begin{equation}\label{level iso}
A_S[m] \iso \underline{L/mL}  ,\qquad \mu_m\iso \underline{\Z/m\Z}
\end{equation}
identifying the Weil pairing on $A_S[m]$  with the pairing \eqref{fixed weil}.

Let  $f:A_0 \to A_1$ be the universal object over $\mathcal{X}_B(N)$.  
For any $m$ prime to $N$, we  denote by 
\[
\mathcal{X}_B(N,m) \to \Spec(\Z[1/m])
\]
the Deligne-Mumford stack whose functor of points assigns to a $\Z[1/m]$-scheme $S$ the groupoid 
whose objects consist of a morphism $S \to \mathcal{X}_B(N)$ together with an $\co_B$-linear isomorphism of group schemes
\[
A_{0/S}[m] \iso \underline{\co_B/m \co_B}.
\]

\begin{proposition}\label{prop:with levels}
For any $m\ge 1$ prime to $N$ there is a commutative diagram 
\begin{equation}\label{q with level}
\xymatrix{
{ \mathcal{X}_B(N,m)  }  \ar[rr] \ar[d] &   & {  \mathcal{C}_L(m)   }  \ar[d]^{\pi_m}  \\
{ \mathcal{X}_B(N)_{ /\Z[1/m]}  }  \ar[rr]_{q_L} &   & {  \mathcal{C}_{ L / \Z[1/m] }   } 
}
\end{equation}
in which all arrows are finite \'etale and relatively representable.  If $m\ge 3$, the top horizontal arrow is an isomorphism of quasi-projective schemes.
\end{proposition}

\begin{proof}
This  amounts to keeping track of \'etale level structure in the proof of Proposition \ref{prop:key morphism}, and we leave everything except for the final claim as an exercise for the reader.

For the final claim,  $\mathcal{C}_L(m)$ is relatively representable and finite over the moduli stack of all principally polarized abelian surfaces with full level $m$ structure over $\Z[1/m]$-schemes.  This moduli stack is a quasi-projective scheme when $m \ge 3$, so the same is true of both $\mathcal{C}_L(m)$ and $\mathcal{X}_B(N,m)$.

As the top horizontal arrow is finite \'etale, it suffices to show that it induces an isomorphism in the complex fiber, which can be done by  examining  the proof of Proposition \ref{prop:complex degree}.  
Indeed, every connected component of the complex fiber of $\mathcal{C}_L(m)$ has the form
$\Gamma(m) \backslash \mathcal{H}^+$ where
\[
\Gamma(m) = \{ g \in \mathrm{U}(L) : g \mbox{ acts trivially on } L/mL\}.
\]
The fiber of the top horizontal arrow over that component is exactly the same, but with $\mathrm{U}(L)$ replaced by $R^1\iso \mathrm{SU}(L)$.    The determinant $\det(g)$ of any $g\in \Gamma(m)$ is an element of $\co_\kk^\times$ satisfying $\det(g) \in 1+ m \co_\kk$, and hence $\det(g) =1$ by the proof of Theorem 5 in Chapter 21 of \cite{mumford-abelian}.  Thus replacing $\mathrm{U}(L)$ by $\mathrm{SU}(L)$ does not change the group $\Gamma(m)$.
\end{proof}

%%%%%%%%%%%%%%%%%%%%%%%%%%%%%%%%%%%%

\section{Arithmetic volumes of unitary Shimura curves}
\label{s:final volumes}

%%%%%%%%%%%%%%%%%%%%%%%%%%%%%%%%%%%%%

Let $\kk$ be a quadratic imaginary field of odd discriminant $\mathrm{disc}(\kk)$.
Suppose $W$ is a $\kk$-hermitian space of signature $(1,1)$ that contains a self-dual $\co_\kk$-lattice
$L\subset W$, and let 
$
\mathcal{C}_L \subset \mathcal{X}_W
$
 be the associated connected component, as in Proposition \ref{prop:components}.
 Let $B$ be the indefinite quaternion algebra associated to $W$ by Proposition \ref{prop:unitary to quaternion}.   As in  \eqref{level}, set
  \[
 N= - \mathrm{disc}(\kk) / \mathrm{disc}(B).
 \]

%%%%%%%%%%%%%%%%%%%%%%%%%%%%%%%%%%%%

\subsection{The anisotropic  case}

%%%%%%%%%%%%%%%%%%%%%%%%%%%%%%%%%%%%%

Assume that $W$ is anisotropic, so that $B$ is a division algebra.
For every prime $p\mid \mathrm{disc}(\kk)$,  define $p^\pdec = \pm p$ as in \eqref{alt p}.
Equivalently, by Remark \ref{rem:invariant isotropic}, 
\[
p^\pdec = 
\begin{cases}
 p & \mbox{if }p\nmid \mathrm{disc}(B)  \\
- p & \mbox{if }p\mid \mathrm{disc}(B)  .
\end{cases}
\]

\begin{theorem}\label{thm:unitary volume}
  The morphism $\mathcal{C}_L \to \Spec(\Z)$ is proper, and the metrized Hodge bundle of the universal $\co_\kk$-abelian surface $A \to \mathcal{C}_L$  has geometric degree
 \[
\deg_\C( \omega^\Hdg_{A / \mathcal{C}_L}   )  
=  \frac{ 1 }{ 12  \cdot  | \co_\kk^\times| }    \prod_{p \mid \mathrm{disc}(\kk)   } (1 + p^\pdec ) 
\]
 and arithmetic volume
\[
\widehat{\mathrm{vol}}\big(  \widehat{\omega}^\Hdg_{A/ \mathcal{C}_L }   \big)  
  =  
 -     \deg_\C(  \omega^\Hdg_{A / \mathcal{C}_L}   )      
\left(
1+  \frac{2\zeta'(-1)}{\zeta(-1)} 
+ \frac{1}{2} \sum_{  p \mid \mathrm{disc}(\kk)     } \frac{ 1- p^\pdec  }{ 1+  p^\pdec } \cdot \log(p)
\right).
\]
In particular, these quantities depend only on the hermitian space $W$, and not on the  connected component $\mathcal{C}_L \subset \mathcal{X}_W$.
\end{theorem}

\begin{proof}
We have constructed in Proposition \ref{prop:key morphism} a finite \'etale  surjection $q_L : \mathcal{X}_B(N) \to \mathcal{C}_L$ of degree $|\co_\kk^\times|$, under which  the metrized Hodge bundle of $A\to \mathcal{C}_L$ pulls back to the metrized Hodge bundle of the intermediate abelian surface $A_R \to \mathcal{X}_B(N)$ determined by a level $N$ Eichler order $R \subset \co_B$.  

Remark \ref{rem:invariant isotropic} implies that  $B$ is a division algebra, and so $\mathcal{X}_B(N)$ is proper.
 The properness of $\mathcal{C}_L$ follows. 
The projection formula for arithmetic intersections, as in Section III.3.1 of \cite{soule92},  implies  that 
\[
\widehat{\mathrm{vol}}\big(  \widehat{\omega}^\Hdg_{A_R/ \mathcal{X}_B(N) }   \big)  
=
| \co_\kk^\times | \cdot 
\widehat{\mathrm{vol}}\big(  \widehat{\omega}^\Hdg_{A/ \mathcal{C}_L }   \big),
\]
and similarly for the geometric degree, and so the desired formulas  follow from those of Theorem \ref{thm:intermediate volume}.
\end{proof}

%%%%%%%%%%%%%%%%%%%%%%%%%%%%%%%%%%%%

\subsection{Pre-log singular Hermitian line bundles}

%%%%%%%%%%%%%%%%%%%%%%%%%%%%%%%%%%%%%

Fix an integer $m\ge 1$.
Suppose $\mathcal{C}^*$ is a regular  Deligne-Mumford stack,  proper and flat over $\Z[1/m]$. 
Let $\partial \mathcal{C}^* \subset \mathcal{C}^*$ be a reduced effective divisor,  flat over $\Z[1/m]$,  whose generic fiber has normal crossings.  Set $\mathcal{C} = \mathcal{C}^* \smallsetminus \partial \mathcal{C}^*$.

The \emph{arithmetic Picard group} $\widehat{\Pic}(\mathcal{C})$ is defined, as before,  as the group  of isomorphism classes of hermitian line bundles on $\mathcal{C}$.   
The \emph{pre-log-singular arithmetic Picard group}  $\widehat{\Pic}(\mathcal{C}^*,\mathscr{D}_\BKK)$ is the group of isomorphism classes of  line bundles on $\omega$ on $\mathcal{C}^*$ equipped with a hermitian metric on $\omega|_{\mathcal{C}}$   that is pre-log singular  along the boundary $\partial \mathcal{C}^*$ in the sense\footnote{In  \cite{BBK} this definition is made under the assumption that $\mathcal{C}^*$ is a  scheme, but  there is no difficulty in extending the definition  to Deligne-Mumford stacks.}  of Definition 1.20 of \cite{BBK}.

\begin{proposition}\label{prop:hermitian restriction}
The natural restriction map
\[
 \widehat{\Pic}( \mathcal{C}^* ,\mathscr{D}_\BKK) \to \widehat{\Pic}( \mathcal{C})
 \]
 is injective, and is an isomorphism if $\partial \mathcal{C}^*=\emptyset$. 
\end{proposition}

\begin{proof}
  The final claim is clear from the definitions, so we only need to prove the injectivity.

Suppose the pre-log singular hermitian line bundle $\widehat{\omega}$ on $\mathcal{C}^*$ becomes trivial after restriction to $\mathcal{C}$.  A choice of  trivializing section   
\[
 s \in H^0( \mathcal{C} , \omega|_\mathcal{C} )
\]
determines  a rational section  of $\omega$ such that $\| s\| =1$ identically on $\mathcal{C}(\C)$, and  whose divisor  is supported on  the boundary $\partial \mathcal{C}^*$.
   We are done if we can show  that this divisor is trivial.  
   As the boundary is  flat over $\Z[1/m]$ by hypothesis, this can be checked  on the complex fiber.

We now view $s$ as a meromorphic section of the holomorphic line bundle $\omega$ on $\mathcal{C}^*(\C)$.
Locally for the orbifold topology, we may choose holomorphic coordinates $z_1,\ldots, z_d$ near a point  $P\in \mathcal{C}^*(\C)$  so that $\partial \mathcal{C}^*(\C)$ is defined by the equation $z_1\cdots z_r =0$, for some $ 0 \le r\le d$, and write 
\[
s= z_1^{e_1} \cdots z_r^{e_r}  s_0
\]
for a nowhere vanishing holomorphic local section  $s_0$.
 By definition of a pre-log singular hermitian metric,   the absolute value of 
\[
\sum_{i=1}^r e_i \log | 1/ z_i |    = \log\| s_0\| 
\]
is bounded by the absolute value of 
$
c  \prod_{i=1}^r \left(  \log \log | 1/ z_i |   \right)^\rho
$
for  positive constants $c$ and $\rho$ on some  polydisc    $| z_1 |,\ldots, | z_d |  < \epsilon$.
  This can only happen if all $e_i=0$, and hence $s$ is holomorphic and nonvanishing  on a neighborhood of  $P \in \mathcal{C}^*(\C)$.
  \end{proof}

\subsection{The isotropic case}

%%%%%%%%%%%%%%%%%%%%%%%%%%%%%%%%%%%%%

Now suppose that $W$ is the unique  isotropic $\kk$-hermitian space of signature $(1,1)$.
Thus $N= - \mathrm{disc}(\kk)$, and  the associated quaternion algebra is $B= M_2(\Q)$ by Remark \ref{rem:invariant isotropic}.

The center of the unipotent radical of any proper parabolic subgroup of $\mathrm{GU}(W)$ is isomorphic to the additive group scheme $\mathbb{G}_a$ over $\Q$.  This implies that the Shimura datum defining $\mathcal{X}_W$ has a unique complete admissible rational polyhedral cone decomposition, and  \cite{pink} gives us  a canonical toroidal compactification of the generic fiber of $\mathcal{X}_W$.
  We want to extend this canonical compactification to the integral model.
  This can be done as in  \cite{lan} and  \cite{howard-unitary-II}, but we will instead follow \cite{mp}

Let $\mathcal{A} \to \Spec(\Z)$ be the  moduli stack of  principally polarized abelian surfaces.
It is a Deligne-Mumford stack  smooth  over $\Z$ of relative dimension $3$. 
Forgetting the $\co_\kk$-action defines a relatively representable, finite, and unramified morphism
$\mathcal{C}_L \to \mathcal{A}$.

The work of Faltings-Chai \cite{faltings-chai} gives us a family of smooth toroidal compactications 
$\mathcal{A} \to \bar{\mathcal{A}}$,   each of which depends on an auxiliary choice of combinatorial data.  
After making such a choice we define $\bar{\mathcal{C}}_L$ as the normalization of 
$\mathcal{C}_L \to \bar{\mathcal{A}}$.
  See \cite[\href{https://stacks.math.columbia.edu/tag/0BAK}{Tag 0BAK}]{stacks-project} for  normalization.

We now add level structure as in \S \ref{ss:level}.
For any $m\ge 1$  define  $\bar{\mathcal{C}}_L(m)$ as the normalization of
$
\mathcal{C}_L(m) \to \bar{\mathcal{C}}_{ L /\Z[1/m] },
$
or, equivalently, as the normalization of $\mathcal{C}_L(m) \to \bar{\mathcal{A}}_{/ \Z[1/m] }$.
In particular, there is a finite and relatively representable morphism
\[
\pi_m : \bar{\mathcal{C}}_L(m) \to \bar{\mathcal{C}}_{ L /\Z[1/m] }
\]
extending the morphism $\pi_m$ of \eqref{q with level}.

\begin{proposition}\label{prop:compactification}
The stack $\bar{\mathcal{C}}_L(m)$ does not depend on the choice of Faltings-Chai compactification $\bar{\mathcal{A}}$ used in its construction, and satisfies the following properties.
\begin{enumerate}
\item
If $m\ge 3$, it is a projective $\Z[1/m]$-scheme. 

\item
It  is regular, and  smooth over $\Z[1/m]$  outside of finitely many points, all contained in the interior $\mathcal{C}_L(m)$ and supported in characteristics dividing $\mathrm{disc}(\kk)$.
 
 \item
 The boundary $\partial  \bar{\mathcal{C}}_L(m) = \bar{\mathcal{C}}_L(m) \smallsetminus \mathcal{C}_L(m)$, 
 with its reduced substack structure, is a Cartier divisor smooth over $\Z[1/m]$. 
\end{enumerate}
\end{proposition}

\begin{proof}
The independence of $\bar{\mathcal{C}}_L(m)$ on the choice of Faltings-Chai compactification follows from Remark 4.1.6 of \cite{mp}, together with the observation above that the Shimura datum defining $\mathcal{X}_W$ admits a unique complete admissible rational polyhedral cone decomposition.
%
%We  sketch another argument for the independence claim.
%Given  another   smooth toroidal compactification 
%$
%\mathcal{A}(1) \to \mathcal{A}^{**}(1),
%$
%one can use the criterion of Theorem 5.7(5) of \cite{faltings-chai} to show that the  morphism
%$\mathcal{C}_L(m)  \to \mathcal{A}(1)_{/\Z[1/m]}$ extends  to 
%\[
%\mathcal{C}^*_L(m)  \to \mathcal{A}^{**}(1)_{/\Z[1/m]}.
%\]
%If  $\mathcal{C}_L^{**}(m)$ denotes the normalization of $\mathcal{A}^{**}(1)_{/\Z[1/m]}$ in $\mathcal{C}_L(m)$,
%then  \cite[\href{https://stacks.math.columbia.edu/tag/035I}{Tag 035I}]{stacks-project} gives us a  morphism 
%\[
%\mathcal{C}^{**}_L(m) \to \mathcal{C}^{*}_L(m)
%\]
% restricting to the identity on $\mathcal{C}_L(m)$.
%  The same reasoning gives a morphism in the other direction with the same property, and so
%  $\mathcal{C}^{**}_L(m) \iso \mathcal{C}^{*}_L(m)$.

For any integer $m\ge 1$,  denote by 
$
\mathcal{A} (m) \to   \Spec(\Z[1/m])
$
the moduli space of  principally polarized abelian surfaces over $\Z[1/m]$-schemes,  equipped with isomorphisms of group schemes as in \eqref{level iso}, compatible with the Weil pairing. 
As explained in Chapter V.5 of \cite{faltings-chai},  if $m\ge 3$  we may choose the compactification of $\mathcal{A}$ in such a way that $\bar{\mathcal{A}}(m)$, defined as the normalization of 
$
\mathcal{A}(m) \to \bar{\mathcal{A}}_{/\Z[1/m]}
$ 
is a projective $\Z[1/m]$-scheme.  
The stack $\bar{\mathcal{C}}_L(m)$ can then be realized as the normalization of the natural map $\mathcal{C}_L(m) \to  \bar{\mathcal{A}}(m)$, yielding  a  finite and relatively representable morphism
$
\bar{\mathcal{C}}_L(m) \to \bar{\mathcal{A}} (m)
$
whose codomain is a projective scheme.  Hence the domain is    a projective scheme.

Part of the assertion of Theorem 1 of \cite{mp}  is that the singularities of $\bar{\mathcal{C}}_L(m)$ are no worse than those of $\mathcal{C}_L(m)$.
To spell this out  in our simple setting,     given a  geometric point $z\to \bar{\mathcal{C}}_L(m)$  contained in the boundary,  one can find a $\Z[1/m]$-scheme $\mathcal{B}_z$  such that the inclusion 
 $ \mathcal{C}_L(m) \to  \bar{\mathcal{C}}_L(m)$ is,  \'etale locally near $z$, isomorphic  to  the  torus embedding
 \[
   \mathbb{G}_{m/\mathcal{B}_z}  \to \mathbb{A}^1_{ / \mathcal{B}_z} .
 \]
 
 The nonsmooth locus of  $\mathcal{C}_L(m)$  is finite over $\Z[1/m]$, and supported in characteristics dividing $\mathrm{disc}(\kk)$.
The finiteness allows us to choose the \'etale neighborhood of  $z$ small enough that it does not meet any nonsmooth points of the interior.  This implies that 
 $\mathbb{G}_{m/\mathcal{B}_z}$ is smooth over $\Z[1/m]$, which implies the smoothness of $\mathcal{B}_z$, which implies the smoothness of $\mathbb{A}^1_{ / \mathcal{B}_z}$, which implies the smoothness of $\bar{\mathcal{C}}_L(m)$ near $z$.   It follows that  $\bar{\mathcal{C}}_L(m)$ is regular,  its nonsmooth points are contained in the interior, and its boundary is a smooth Cartier divisor.
  \end{proof}

%\begin{remark}
%As the reader can see, all of the nontrivial information about $\mathcal{C}_L^*(m)$ contained in Proposition \ref{prop:compactification} relies on the results of Madapusi Pera \cite{mp}, which apply to all Hodge type Shimura varieties.   
%One could instead appeal to the results of \cite{lan} to construct a toroidal compactification of  $\mathcal{C}_L(m)$ over $\Z[ 1/ (m\cdot \mathrm{disc}(\kk) ) ]$.
%Extending the constructions of \cite{lan} to account for the primes dividing $\mathrm{disc}(\kk)$ doesn't require any new ideas, and is done in \cite{howard-unitary-II} when $m=1$.
%However, appealing to  \cite{mp} gives a more direct route to our goal, so that is the approach we have chosen.
%\end{remark}

Pullback of hermitian line bundles  defines a homomorphism 
\[
\pi_m^* : \widehat{\Pic} (  \mathcal{C}_{L / \Z[1/m]}   )  \to \widehat{\Pic}(   \mathcal{C}_L(m) ).
\]
By Proposition 7.5 of \cite{BKK}, and recalling  Proposition \ref{prop:hermitian restriction},  this restricts to a homomorphism
\[
\pi_m^* : \widehat{\Pic} (  \bar{\mathcal{C}}_{L / \Z[1/m]}  , \mathscr{D}_\BKK  ) 
 \to \widehat{\Pic}(  \bar{ \mathcal{C}}_L(m)  , \mathscr{D}_\BKK).
\]

Now suppose $m\ge 3$, so that  $\bar{\mathcal{C}}_L(m)$ is a projective scheme, 
and recall the arithmetic Chern class map
\[
\widehat{\Pic}( \bar{\mathcal{C}}_L(m)  , \mathscr{D}_\BKK) 
\to  \widehat{\mathrm{CH}}^1( \bar{\mathcal{C}}_L(m) ,\mathscr{D}_\BKK) 
\]
 from  (1.13) of \cite{BBK}.
 The codomain here is  the arithmetic Chow group with pre-log-log forms with respect to the boundary $\partial \bar{\mathcal{C}}_L(m)$,   as in \S 1 of \cite{BBK} and \S 7 of \cite{BKK}.
Given  a pre-log-log hermitian line bundle
\[
\widehat{\Omega} \in \widehat{\Pic}( \bar{\mathcal{C}}_L(m) , \mathscr{D}_\BKK),
\]
we may form the self-intersection 
\[
\widehat{\Omega} \cdot \widehat{\Omega} \in
  \widehat{\mathrm{CH}}^2( \bar{\mathcal{C}}_L(m) ,\mathscr{D}_\BKK) \otimes_\Z\Q
\]
of its arithmetic Chern class.  If we set 
$
\R_m= \R / \sum_{p\mid m}  \Q  \log(p) ,
$
then Remark 1.19 of \cite{BBK} provides us with an  arithmetic degree
\[
\widehat{\deg} :  \widehat{\mathrm{CH}}^2( \bar{\mathcal{C}}_L(m) ,\mathscr{D}_\BKK) \to  \R_m,
\]
and we define the \emph{arithmetic volume}
$
\widehat{\vol}(\widehat{\Omega} )  =
 \widehat{\deg}  \big(  \widehat{\Omega} \cdot \widehat{\Omega}  \big) \in \R_m.
$

As in \S 6.3 of \cite{BBK},   any pre-log-log hermitian line bundle
\[
\widehat{\Omega} \in \widehat{\Pic}( \bar{\mathcal{C}}_L , \mathscr{D}_\BKK)
\]
has an   \emph{arithmetic volume}    
\[
  \widehat{\vol} \big(\widehat{\Omega} \big)  =    \mil_{m \ge 3} 
\frac{  \widehat{\vol} \big( \pi_m^* \widehat{\Omega} \big)}{  \mathrm{deg}(\pi_m)  }   \in \mil_{m\ge 3}  \R_m = \R,
\]
where the limit is with respect to the natural maps $\R_{m'} \to \R_m$ for $m' \mid m$.

%We are now in a position to prove the analogue of Theorem \ref{thm:unitary volume} when $\mathcal{C}_L$ is noncompact.  The proof follows the same lines, except that the volume calculations of Kudla-Rapoport-Yang on compact quaternionic Shimura curves are  replaced by the calculations of K\"uhn (and  Kramer, in independent unpublished work) on modular curves.

\begin{theorem}\label{thm:compactified unitary volume}
The metrized Hodge bundle of the universal $\co_\kk$-abelian surface $A \to \mathcal{C}_L$ lies in the subgroup
\[
\widehat{\Pic} ( \bar{\mathcal{C}}_L,\mathscr{D}_\BKK ) \subset \widehat{\Pic}(\mathcal{C}_L )
\]
of Proposition \ref{prop:hermitian restriction}.  
It  has geometric degree
 \[
\deg_\C( \omega^\Hdg_{A / \mathcal{C}_L}   )  
=  \frac{ 1 }{ 12  \cdot  | \co_\kk^\times| }    \prod_{p \mid \mathrm{disc}(\kk)   } (1 + p ) 
\]
 and arithmetic volume
\[
\widehat{\mathrm{vol}}\big(  \widehat{\omega}^\Hdg_{A/ \mathcal{C}_L }   \big)  
  =  
 -     \deg_\C(  \omega^\Hdg_{A / \mathcal{C}_L}   )      
\left( 
1+  \frac{2\zeta'(-1)}{\zeta(-1)} 
+ \frac{1}{2} \sum_{  p \mid \mathrm{disc}(\kk)     } \frac{ 1- p  }{ 1+  p } \cdot \log(p)
\right).
\]
In particular, these are independent of  the  connected component $\mathcal{C}_L \subset \mathcal{X}_W$.
\end{theorem}

\begin{proof}
As in Remark \ref{rem:split intermediate}, fix $\co_B \iso M_2(\Z)$, and identify the quaternionic Shimura curve $\mathcal{X}_B(N)$ with the open modular curve of level $\Gamma_0(N)$. 
The universal isogeny of $\co_B$-abelian surfaces takes the form
\begin{equation}\label{modular splitting}
A_0 \iso E_0 \times E_0 \to E_1\times E_1\iso A_1
\end{equation}
for elliptic curves $E_0, E_1 \to \mathcal{X}_B(N)$.   
For any integer $m \ge 3 $ relatively prime to $N$, the finite \'etale cover
\[
 \mathcal{X}_B(N,m) \to \mathcal{X}_B(N)_{/\Z[1/m]}
\]
of Proposition \ref{prop:with levels} classifies  full level $m$ structures  on $E_0$, and there is a canonical isomorphism
\begin{equation}\label{modular iso}
 \mathcal{C}_L(m) \iso  \mathcal{X}_B(N,m) 
\end{equation}
under which the universal $\co_\kk$-abelian surface $A\to \mathcal{C}_L(m)$ on the left is identified with  the intermediate abelian surface $A_R=E_1 \times E_0$ of Remark \ref{rem:split intermediate}.

It follows from  calculations of K\"uhn and Kramer, see especially Remark 4.10 and Corollary 6.2 of \cite{kuhn}, that 
the metrized Hodge bundle of the universal elliptic curve $E_0 \to \mathcal{X}_B(N,m)$ lies in the subgroup
\begin{equation}\label{unitary log subgroup}
 \widehat{\Pic}(  \bar{\mathcal{X}}_B(N,m) , \mathscr{D}_\BKK )   \subset  \widehat{\Pic}(  \mathcal{X}_B(N,m) ) .
\end{equation}
Here $ \mathcal{X}_B(N,m) \hookrightarrow \bar{\mathcal{X}}_B(N,m)$ is the Deligne-Rapoport compactification, constructed as a moduli space of generalized elliptic curves with level structure.  Moreover,  
 \begin{equation}\label{kuhn degree}
  \deg_\C( \omega^\Hdg_{E_0 / \mathcal{X}_B(N,m)  })   
=    \frac{ d(m)  }{24 }  
\end{equation}
  and 
\begin{equation}\label{kuhn volume}
   \widehat{\vol} \big(  \widehat{\omega}^\Hdg_{E_0/  \mathcal{X}_B(N,m) }    \big)  
   =     
  \frac{ d(m) }{2}      \cdot 
 \left(  \frac{ \zeta(-1)}{2}  +     \zeta'(-1)    \right) \in \R_m
\end{equation}
where 
 \[
d(m) =   \varphi (m)   \cdot  [ \mathrm{SL}_2(\Z)   : \Gamma_0(N) \cap \Gamma(m) ] 
\]
is the degree\footnote{Our $d(m)$ is twice the integer $d$ appearing in Corollary 6.2 of \cite{kuhn}, because   K\"uhn's  $\Gamma_0(N)\cap \Gamma(m)$ is a subgroup of $\mathrm{PSL}_2(\Z)$, not  $\mathrm{SL}_2(\Z)$.
}  of the forgetful morphism from  $\mathcal{X}_B(N,m)$ to the moduli stack of all elliptic curves.

We have  constructed  $\bar{\mathcal{C}}_L(m)$ as  the normalization of a morphism $\mathcal{C}_L(m) \to \bar{\mathcal{A}}$ to  a toroidal compactification of the Siegel $3$-fold, and  must compare this compactification of $\mathcal{C}_L(m) \iso \mathcal{X}_B(N,m)$ with the compactification appearing in \eqref{unitary log subgroup}.

\begin{lemma}
The  isomorphism \eqref{modular iso} extends (necessarily uniquely) to an isomorphism
\[
\bar{\mathcal{C}}_L(m) \iso  \bar{\mathcal{X}}_B(N,m).
\]
\end{lemma}

\begin{proof}
It follows from the moduli interpretation of $\bar{\mathcal{X}}_B(N,m)$ that the elliptic curves $E_0$ and $E_1$ over $\mathcal{X}_B(N,m)$ extend to smooth group schemes over $\bar{\mathcal{X}}_B(N,m)$ with toric degeneration along the boundary.  In other words, they extend to semi-abelian schemes, and hence the same is true of the product $A_R=E_1\times E_0$.  Theorem 5.7(5) in Chapter IV of \cite{faltings-chai} therefore implies that the map
\[
\mathcal{X}_B(N,m) \to \bar{\mathcal{A}}
\]
defined by the polarized abelian surface $A_R \to \mathcal{X}_B(N,m)$ extends uniquely to a morphism
\[
\bar{\mathcal{X}}_B(N,m) \to \bar{\mathcal{A}}.
\]

By the universal property of normalization  \cite[\href{https://stacks.math.columbia.edu/tag/035I}{Tag 035I}]{stacks-project},
there is a unique arrow $i$ making the diagram 
\[
\xymatrix{
{  \mathcal{C}_L(m) } \ar[r] \ar[d] & {  \bar{\mathcal{X}}_B(N,m)  }  \ar[d] \\
{  \bar{\mathcal{C}}_L(m)}  \ar[r] \ar[ur]^{i} & { \bar{\mathcal{A}}}
}
\]
commute.
The diagonal arrow $i$ restricts to an isomorphism of generic fibers, because this restriction is a birational map between smooth proper curves.  
Moreover, the diagonal arrow is quasi-finite: the fiber over a closed point in the interior  $ \mathcal{X}_B(N,m)$ consists of a single point, while the fiber over a closed point  in characteristic $p$ of the boundary $\partial \bar{\mathcal{X}}_B(N,m)$ is contained in the mod $p$ fiber of  $\partial \bar{\mathcal{C}}_L(m)$, which is finite.  Being proper and quasi-finite, $i$ is  finite by  \cite[\href{https://stacks.math.columbia.edu/tag/02OG}{Tag 02OG}]{stacks-project}.

It follows that we may cover $ \bar{\mathcal{X}}_B(N,m) $ by open affines $U=\Spec(R)$ in such a way that $i^{-1}(U)= \Spec(R')$ is affine, and $R\to R'$ is a finite morphism of normal domains inducing an isomorphism on fraction fields.  Any such $R\to R'$ is an isomorphism, and hence $i$ is itself an isomorphism.
\end{proof}

The first isomorphism in \eqref{modular splitting} determines an isomorphism of hermitian line bundles
\[
 \widehat{\omega}^\Hdg_{A_0/ \mathcal{C}_L(m)}
 \iso 
 \widehat{\omega}^\Hdg_{E_0/ \mathcal{C}_L(m) }\otimes  \widehat{\omega}^\Hdg_{E_0/ \mathcal{C}_L(m)} ,
\]
while the commutativity of \eqref{q with level} implies 
\[
 \mathrm{deg}(\pi_m) = \frac{  |\co_\kk^\times| \cdot  d(m) }{  [ \mathrm{SL}_2(\Z) : \Gamma_0(N) ]  } .
\]
Thus from the formulas of K\"uhn and Kramer cited above, one obtains formulas for the complex degree and arithmetic volume of the metrized Hodge bundle of $A_0=E_0\times E_0 \to \mathcal{X}_B(N,m)$, exactly analogous to the formulas  of Kudla-Rapoport-Yang cited in the proof of Theorem \ref{thm:intermediate volume}.

With these formulas in hand, the  proofs of Theorem \ref{thm:intermediate volume} and Theorem \ref{thm:unitary volume} extend to the compactification of $\mathcal{X}_B(N,m) \iso \mathcal{C}_L(m)$ without change, and show that the metrized Hodge bundle of the universal $\co_\kk$-abelian surface $A \to \mathcal{C}_L(m)$, which is identified  with the metrized Hodge bundle of the intermediate abelian surface $A_R \to \mathcal{X}_B(N,m)$,   lies in the subgroup
\[
 \widehat{\Pic}(  \bar{\mathcal{C}}_L(m) , \mathscr{D}_\BKK )   \subset  \widehat{\Pic}(  \mathcal{C}_L(m) ) 
\]
  and satisfies
\[
  \deg_\C( \omega^\Hdg_{A / \mathcal{C}_L(m)  } )
=    \frac{  \mathrm{deg}(\pi_m)   }{12 \cdot  |\co_\kk^\times|  }  \cdot \prod_{ p\mid N} (1+p)
\]
 and 
\begin{align*}
\widehat{\mathrm{vol}}  (  \widehat{ \omega}^\Hdg_{A/ \mathcal{C}_L(m) }   ) 
   =  - \deg_\C(  \omega^\Hdg_{A/ \mathcal{C}_L(m)}   )  \cdot 
\left(
1+ \frac{2\zeta'(-1)}{\zeta(-1)} 
+  \sum_{ p\mid N} \frac{1-p}{1+p} \cdot  \frac{ \log(p)}{2}
\right),
\end{align*}
where the latter equality holds in  $\R_m= \R / \sum_{p\mid m}  \Q  \log(p)$.

%The only subtlety  is that the hermitian line bundles of \eqref{vertical bundles}, after pullback to $\mathcal{X}_B(N,m) \iso \mathcal{C}_L(m)$,   must be extended to $\mathcal{C}^*_L(m)$ 
%by taking the Zariski closures of the underlying vertical divisors $\mathcal{F}_p$ and $\mathcal{V}_p$ on $\mathcal{C}_L(m)$.
%For the proof of  Theorem \ref{thm:unitary volume}  to go through one needs to know that the intersection of the Zariski closures is still the reduced supersingular locus;
%in other words, that the Zariski closures cannot intersect at the boundary. 
%But this is clear, as  such an intersection would contradict the smoothness of (the reduction modulo $p$ of) $\mathcal{C}_L^*(m)$ near the boundary  established in Proposition \ref{prop:compactification}.

To complete the proof of Theorem \ref{thm:compactified unitary volume}, one upgrades this to an equality in $\R$ by
choosing relatively prime integers $m_1 , m_2 \ge 3$, both prime to $N$.  The $\Q$-linear independence of $\{ \log(p) : p \mbox{ prime}\}$ implies that 
\[
\R = \R_{m_1} \times_{\R_{m_1m_2}} \R_{m_2} ,
\]
 and so
all of the stated properties of the metrized Hodge bundle of $A \to \mathcal{C}_L$ follow by combining  the corresponding properties of $A\to \mathcal{C}_L(m_i)$ proved above.
\end{proof}

\bibliographystyle{alpha.bst}
%\bibliography{volumes}

\def\cprime{$'$}

\end{document}